\documentclass[a4paper,fleqn]{cas-sc}
\usepackage[numbers]{natbib}

\def\tsc#1{\csdef{#1}{\textsc{\lowercase{#1}}\xspace}}
\tsc{WGM}
\tsc{QE}
\tsc{EP}
\tsc{PMS}
\tsc{BEC}
\tsc{DE}
\usepackage{graphics,graphicx,color}
\usepackage{rotating}
\usepackage{amssymb,amsmath}
\usepackage{eucal}
\usepackage{subfigure,url}
\usepackage{algorithm}
\usepackage[noend]{algpseudocode}
\usepackage{epsfig,epstopdf}
\begin{document}

	\newcommand\cero{\boldsymbol{0}}
	\newcommand{\norm}[1]{\ensuremath{\left\|#1\right\|}}
	\newcommand{\fracd}[2]{\displaystyle
		{\frac{{\displaystyle{#1}}}{{\displaystyle{#2}}}}}
	\newcommand\vdiv{\mathop{\mathrm{div}}\nolimits}
	\newcommand\curl{\mathop{\mathrm{curl}}\nolimits}
	
	\newcommand{\br}{\boldsymbol{r}}
	
	\newcommand\bcurl{\mathop{\mathbf{curl}}\nolimits}
	\newcommand\rot{\mathop{\textrm{rot}}\nolimits}
	\def\Curl{\mathop{\mathsf{Curl}}\nolimits}
	\newcommand{\rota}{\rot_{\mathbf{a}}}
	
	\newcommand{\abs}[1]{\ensuremath{\left|#1\right|}}
	
	\newcommand{\eps}{\ensuremath{\varepsilon}}
	\newcommand{\beps}{\boldsymbol{\varepsilon}}
	
	\newcommand{\bF}{\ensuremath{\boldsymbol{F}}}
	\newcommand{\bG}{\boldsymbol{G}}
	\newcommand{\bI}{\ensuremath{\mathbf{I}}}
	
	\newcommand{\uu}{\ensuremath{\overline{u}}}
	
	\newcommand{\Lap}{\mathrm{\Delta}}
	\newcommand{\Om}{\ensuremath{\Omega}}
	
	\newcommand{\K}{\ensuremath{\mathcal{K}}}
	
	\newcommand{\T}{\ensuremath{\mathcal{T}_h}}
	\newcommand{\C}{\ensuremath{\mathcal{C}}}
	\newcommand{\mS}{\ensuremath{\mathcal{S}}}
	\newcommand{\bu}{\boldsymbol{u}}
	\newcommand{\bn}{\boldsymbol{n}}
	\newcommand{\bq}{\boldsymbol{q}}
	\newcommand{\bp}{\boldsymbol{p}}
	\newcommand{\bomega}{\boldsymbol{\omega}}
	\newcommand{\bt}{\boldsymbol{t}}
	\newcommand{\bv}{\boldsymbol{v}}
	\newcommand{\bx}{\boldsymbol{x}}
	\newcommand{\bdiv}{\boldsymbol{\vdiv}}
	\newcommand{\bw}{\boldsymbol{w}}
	\newcommand{\bz}{\boldsymbol{z}}
	\newcommand{\ff}{\boldsymbol{f}}
	\renewcommand{\gg}{\boldsymbol{g}}
	\renewcommand{\ss}{\boldsymbol{s}}
	\newcommand{\bfi}{\ensuremath{\boldsymbol{\varphi}}}
	
	\def\dt{{\,\mathrm{d}t}}
	\def\dx{{\,\mathrm{d}\bx}}
	\def\x{{\,\mathrm{x}\x}}
	\def\dxdt{{\,\mathrm{d}\bx\mathrm{d}t}}
	\def\dxds{{\,\mathrm{d}\bx\mathrm{d}\s}}
	\def\ds{{\,\mathrm{d}\s\,}}
	\def\pt{\partial_t}
	\def\ps{\partial_s}
	
	\def\ph{\varphi}
	\def\s{\sigma}
	\def\t{\theta}
	\newcommand{\A}{\ensuremath{\mathcal{A}}}
	\def\cC{\mathcal{C}}
	\def\cD{\mathcal{D}}
	\def\cE{\mathcal{E}}
	\def\cR{\mathcal{R}}
	\def\cP{\mathcal{P}}
	\def\cQ{\mathcal{Q}}
	\def\cS{\mathcal{S}}
	\def\RT{\mathbf{RT}}
	\def\ND{\mathbf{ND}}

	\def\Ra{\mathrm{Ra}}
	\def\bD{\mathbf{D}}
	\def\bK{\mathbb{K}}
	\def\bV{\mathbf{V}}
	\def\Beta{\boldsymbol{\eta}}
	\def\bzeros{\mathbf{0}}
	\def\bH{\mathbf{H}}
	\def\bbH{\mathit{H}}
	\def\bW{\mathbf{W}}
	\def\bZ{\mathbf{Z}}
	\def\V{\mathrm{V}}
	\def\L{\mathrm{L}}
	\def\bL{\mathbf{L}}
	
	\def\cT{\mathcal{T}}
	\def\W{\mathcal{W}}
	\def\U{\mathcal{U}}
	\def\P{\mathcal{P}}
	\def\CVh{\mathcal{V}_h}
	\def\CQh{\mathcal{Q}_h}
	\def\bbH{\mathbb{H}}
	\def\bsi{\boldsymbol{\sigma}}
	\def\btau{\boldsymbol{\tau}}
	\def\tr{\mathop{\mathrm{tr}}\nolimits}
	\newcommand{\Grad  }{\nabla}
	\newcommand{\bpsi}{\boldsymbol{bpsi}}
	\def\bpsi{\mathbf{\psi}}

	\def\qin{{\quad\hbox{in}\quad}}
	\def\qon{{\quad\hbox{on}\quad}}
	
	\newtheorem{Remark}{Remark}[section]
	\numberwithin{equation}{section}

	\renewcommand\H{\mathrm{H}}
	\renewcommand\L{\mathrm{L}}
	\newcommand\Q{\mathrm{Q}}
	\newcommand\Z{\mathrm{Z}}
	\newcommand\R{\mathbb{R}}
	\newcommand\X{\mathrm{X}}
	
	\newcommand{\bphi}{\boldsymbol{\phi}}
	\newcommand{\cblue}{}
	\newcommand{\cred}[1]{\textcolor{red}{#1}}
	\newcommand{\cgray}[1]{\textcolor{cgray}{#1}}
	
	\def\Div{\mbox{{\rm Div}}}
	\newcommand*\nnorm[1]{|\!|\!| #1 |\!|\!|}
	
	\newcommand{\set}[1]{\left\{#1\right\}}
	
	\newcommand\disp{\displaystyle}
	
	\newcommand\qand{\quad\hbox{and }\quad}
	
	\renewcommand\O{\Omega}
	\newcommand\G{\Gamma}
	
	\renewcommand\H{\mathrm{H}}
	\renewcommand\L{\mathrm{L}}
	
	\newcommand\LO{\L^2(\O)}
	\newcommand\LOO{\L_{0}^2(\O)}
	
	\newcommand\hdivO{{\H(\mathrm{div};\O)}}
	\newcommand\HO{\H^1(\O)}
	\newcommand\HCUO{\H_{0}^1(\O)}
	\newcommand\HsO{\H^s(\O)}
	\newcommand\HusO{\H^{1+s}(\O)}
	\newcommand\nxn{n\times n}
	\renewcommand\S{\Sigma}
	\renewcommand\O{\Omega}
	\newcommand\btheta{\boldsymbol{\theta}}
	
	\newtheorem{Theorem}{Theorem}
	\newtheorem{Proposition}{Proposition}
	\newtheorem{Lemma}{Lemma}
	\newtheorem{Definition}{Definition}
	\newenvironment{proof}{\noindent{\it Proof.}}{\hfill$\square$}
	\let\WriteBookmarks\relax
	\def\floatpagepagefraction{1}
	\def\textpagefraction{.001}
	\shorttitle{Nonlinear predator-prey cross-diffusion--fluid system with two chemicals}
	\shortauthors{M. Bendahmane,  F. Karami, D. Meskine, J. Tagoudjeu and M. Zagour}
	
	\title [mode = title]{Mathematical analysis and multiscale derivation of a nonlinear predator-prey cross-diffusion--fluid system with two chemicals}
	
	
	\author{Mostafa Bendahmane$^1$}
	\fnmark
	\ead{mostafa.bendahmane@u-bordeaux.fr}
	\address{$1$ Institut de Math\'ematiques de Bordeaux, Universit\'e de Bordeaux, 33076 Bordeaux Cedex, France}
	
	\author{Fahd Karami$^2$}
	\fnmark
	\ead{fa.karami@uca.ma}
	\address{$2$ École Supérieure de Technologie d'Essaouira, Université Cadi Ayyad, B.P. 383 Essaouira El Jadida, Essaouira, Morocco}
	
	\author{Driss Meskine$^2$}
	\fnmark
	\ead{dr.meskine@uca.ac.ma }
	
	\author{Jacques Tagoudjeu$^3$}
	\ead{jacques.tagoudjeu@univ-yaounde1.cm}
	\address{$3$ École Nationale Supérieure Polytechnique de Yaoundé,
		Universite de Yaoundé I, B.P 8390 Yaoundé, Cameroun}
	
	\author{Mohamed Zagour$^4$}
	\fnmark
	\ead{zagourmohamed@gmail.com}
	\address{$4$ Euromed Research Center, Euromed University of Fes, Rte Principale Fès Meknès, 30000 Fès, Morocco}
	
	\begin{abstract}
		A nonlinear cross-diffusion--fluid system with chemical terms describing the dynamics of predator-prey living in a Newtonian fluid is proposed in this paper. The existence of a weak solution for the proposed macro-scale system is proved based on the Schauder fixed-point theory, a priori estimates, and compactness arguments. The proposed system is derived from the underlying description delivered by a kinetic-fluid theory model by a multiscale approach. Finally, we discuss the computational results for the proposed macro-scale system in two-dimensional space.
	\end{abstract}
	
	\begin{keywords}
	Chemical cross-diffusion--fluid; Kinetic--fluid theory;  Schauder fixed-point theory; Pattern formation; Finite-volume method; Finite-element method.
	\end{keywords}
	
	\maketitle

	\section{Introduction}
	
	As it is known, cross-diffusion mathematical models have been helpful to predict many interesting features such as pattern-formation, dynamics segregation phenomena, and competition between interacting populations. 
	Several models have been proposed and studied in the literature for competing species living outside the fluid medium. Originally, the classic ecological models began with the study of two interacting species (see, e.g., \cite{Jun10,SK97} for more details). Next, some cross-diffusion models of three and multiple interacting species \cite{ABS15,CDJ18,GJF15,KH94,MY95}  were proposed. The author in \cite{BKZ18} proposed a model with two interacting species living in a stationary fluid governed by the augmented Brinkman system. Recently, the author in \cite{ABKMZ20} generalized the aforesaid model to a nonlocal cross-diffusion with multiple species living in a Newtonian fluid governed by the incompressible Navier-Stokes.  Indeed, the motivation comes from the fact that many species are living in a fluid. Consequently, their dynamic is affected by the presence of the fluid. Compare with the previously cited articles, in the present paper we propose a nonlinear predator-prey cross-diffusion--fluid with two chemicals. The predator and prey species present the ability to orientate their movement towards the concentration of the chemical secreted by the other species. The problem is presented as a system of two parabolic equations describing the evolution of the predator and prey species and two elliptic equations for the concentration of the chemicals coupled with the incompressible Navier-Stokes.

	In order to state our problem, let consider $\Omega\in \mathbb{R}^d,$ $d = 1,2, 3$, a simply connected domain saturated with a Newtonian incompressible fluid, where also predator and prey species and two chemical substances are present. The physical scenario of interest can be described by the following nonlinear macro-scale system in $ T := (0, T)\times\Omega $ for a fixed time $T > 0$ written in a non-dimensional form
	\begin{equation}\label{CrossDiff}
		\left\{
		\begin{array}{l}
			\displaystyle \partial_t n_{1}+U\cdot\nabla n_1- \vdiv\big(d_{1}(n_1)\nabla n_1\Big) +\vdiv \bigl(\chi_1(n_1)\nabla 
			w_1 \bigl)=F_1(n_1,n_2),\\\\
			\displaystyle \partial_t n_{2}+U\cdot\nabla  n_2- \vdiv\big(d_{2}(n_2)\nabla n_2\big)  + \vdiv \bigl(\chi_2(n_2)\nabla 
			w_2 \bigl)=F_2(n_1,n_2),
			\\ \\
			U\cdot\nabla  w_1-\Delta w_1+\alpha_{1}\,w_1=\beta_1 n_2,\\\\
			U\cdot\nabla  w_2-\Delta w_2+\alpha_{2}\,w_2=\beta_2 n_1,\\\\
			\displaystyle  \partial_t U -\nu \Delta U+ k(U\cdot \nabla)U+\nabla p+Q(n_1,n_2) \nabla\phi = \bzeros,  \; \;  \vdiv U=0.
		\end{array}
		\right.\end{equation}
	We augment our proposed macro-scale system with the following boundary conditions
	\begin{equation}\label{BC}
		\Big(d_{i}(n_i)\nabla n_i-\chi_i( n_i)\nabla w_i\Big)\cdot\eta=0,\qquad \nabla w_i\,\eta=0,\qquad U=\bzeros,\qquad \text{on}\;\;\; \Sigma_T=(0,T]\times\partial\Omega
	\end{equation}
	and the initial conditions
	\begin{equation}\label{IC}
		n_i(t=0,x)=n_{i,0}(x),\qquad U(t=0,x)=U_0(x)\qquad \text{for}\;\; x\in \Omega
	\end{equation}
	for $i=1,2$.
	Here $n_1$ and $n_2$ denote population densities of the predator and the prey, respectively, $w_1$ and $w_2$ represent concentrations of the (chemical) signals produced by $n_2$ and $n_1$ respectively; $U$ is the fluid velocity, $p$ is the fluid pressure; $d_1$ and $d_2$ are the nonlinear diffusion functions; $\chi_i$ are the  nonlinear  tactic functions; $\alpha_i,\, \beta_i$ for $i = 1,2$ are positive constants. 
	
	Tactic coefficients play a major role from a modeling point of view. Indeed, one can find in nature that the movement of biological species is oriented by chemical gradients, where the predator moves towards the prey. Different types of situations can occur depending on the ability of predator and prey to direct their movement towards these chemical gradients. A typical example is the following: the tactic coefficients: $ \chi_1> 0$ and $\chi_2 < 0$ model the situation where the prey avoids the predator by moving away from its signal gradient, while the predator follows the prey by following a higher concentration of the chemical $w_2$.

	Finally, $F_1$ and $F_2$ are Lotka-Voltera reaction terms given by
	\begin{equation}\label{LV}
		F_1(n_1,n_2)=n_1(a_1-b_1n_1-c_1n_2),\;\;F_2(n_1,n_2)=n_2(a_2-c_2n_2+b_2n_1),
	\end{equation}
	where $a_1,\, a_2,\, b_1,\, b_2,\, c_1$ and $c_2$ are the positive coefficients of intra-specific competition and inter-specific competition.
	Let us mention that macro-scale system \eqref{CrossDiff} indicates that the predator is attracted by the chemical signal $w_1$ of the prey $n_1$, while the prey is repelled by the chemical signal $w_2$ produced by the predator. Note that the equations for prey and predator odors are elliptical rather than parabolic. This is justified in cases where odor diffusion occurs on a much faster time scale than the movement of individuals, which is reasonable in a variety of ecological settings. Note that we refer to $w_1$ and $w_2$ as chemical signals which can be interpreted more generally as potentials representing the possibility of an animal being detected from a distance, for example by visual means. However, for example, these quantities can model chemical odors. The coupling in our system \eqref{CrossDiff} appears through the convection term $U\cdot\nabla n_i$, $U\cdot\nabla w_i$ and the external force $Q(n_1,n_2)\nabla \phi$.

	In the absence of the fluid i.e. $(U=\bzeros)$, system \eqref{CrossDiff} reduces to chemotaxis chemicals system. Among others in \cite{NT19} the authors proved global existence and asymptotic behavior of solutions. Systems of two biological species with kinetic interaction have been considered in \cite{TW12}, where the stability of homogeneous steady states is obtained for one chemical (see \cite{BW16,BLM16,TW14}).
	Competitive systems of two biological species and a chemical with non-constant coefficients have been considered in \cite{TW19} where the authors establish sufficient conditions for the existence of solutions and its asymptotic dynamics. For the one species case with time and space dependence coefficients and growth term we refer to the reader to \cite{TW17}.
	Moreover, systems of two biological species with chemotactic abilities have been studied. For instance, in \cite{CNT18} the competitive system is studied and the global existence and asymptotic behavior are obtained for positive and bounded initial data. While in \cite{ZX17}, the reduced system is studied for constant coefficients in the competitive case.
	Several numerical methods have been used for solving nonlinear predator-prey and competitive systems. For instance, in \cite{YJ19} authors solve a two species system using a moving mesh finite elements in one dimension. Also, a particle method and the the meshless method of the Generalized Finite Differences have been applied respectively in \cite{GLS09,BGGNU21}.

	In this paper we address a multiscale derivation approach of the proposed macro-scale model from kinetic theory model based on the micro-macro decomposition method. We start by rewriting the kinetic theory model as a coupled system of microscopic and macroscopic equations. Next, the proposed macro-scale model is derived by low order asymptotic expansions in terms of a small parameter. This approach has been applied to the micro-macro application in different fields. For instance, chemotaxis phenomena \cite{BBNS15}, a time-dependent SEIRD reaction diffusion \cite{Z21}, and patterns formation induced by cross-diffusion in a fluid~\cite{ABKMZ20,BKZ18}. Note that this technique motivated the design numerical tools that preserve the asymptotic property \cite{JI99,KL98}. Specifically, these methods design the uniform stability and consistency of numerical schemes in the limit along the transition from kinetic regime to macroscopic regime.

	This paper is organized as follows: Section \ref{Sec2} is devoted to establish the existence of weak solutions of the proposed nonlinear cross-diffusion--fluid system \eqref{CrossDiff}. The proof is based on Schauder fixed-point theory, a priori estimates, and compactness arguments. In Section \ref{Sec3}, we present our kinetic--fluid theory model and its properties. According to a multiscale approach based on the micro-macro decomposition method, we obtain an equivalent micro-macro formulation. This leads to derive our proposed macro-scale system \eqref{CrossDiff}. 
	In Section \ref{Sec4}, we investigate the computational analysis of cross-diffusion--fluid system \eqref{CrossDiff} in two dimensional space. We provide several numerical simulations with two cases: in the first case, we ignore the fluid effect ($U=\bzeros$) by using finite volume method. in the second one, we consider the full system \eqref{CrossDiff} using finite element method.

	\section{Mathematical analysis}\label{Sec2}
     
     Let $\Omega$ be a bounded, {open subset} of ${{\R}}^d$, $d=2,3$ with a
     smooth boundary $\partial \Omega$ and $|\Omega|$ is the Lebesgue measure of $\Omega$. 
     We denote by $\H^1(\Omega)$ the Sobolev space of functions $u:\Omega \to \R$
     for which $n\in \L^2(\Omega)$ and $ \nabla n \in \L^2(\Omega ;{{\R}}^d)$.
     For $1\leq p \leq +\infty$, $\parallel \cdot \parallel_{\L^p(\Omega)}$ 
     is the usual norm in $\L^p(\Omega)$.
     If $X$ is a Banach space, $a<b$ and $1\leq p \leq +\infty$, 
     $\L^p(a,b;X)$ denotes the space of all measurable functions 
     $n~:~(a,b) \longrightarrow X$ such that $\parallel n(\cdot)\parallel_X$ 
     belongs to $\L^p(a,b)$. 
     
     Now, we introduce basic spaces in the study of the Navier-Stokes equation. Let the spaces $\mathcal{V}, \:\bV$ and $\bH$ defined as:
     $$\mathcal{V}= \{U\in \mathcal{D}(\Omega), \;  \vdiv U=0\} ,\; \bV= \overline{\mathcal{V}}^{\H^1_0(\Omega)},\; \bH= \overline{\mathcal{V}}^{\L^2(\Omega)}. $$
     The coupled system of interest (\ref{CrossDiff}) can be written as for $i=1,2$
     \begin{equation}\label{CrossFF}
     	\left\{
     	\begin{array}{ll}
     		\displaystyle
     		\partial_t n_i+ U\cdot\nabla n_i-  \vdiv\Big(d_i({n_i} ) \nabla n_i+\chi_i(n_i)\nabla w_i \Big)= F_i(n_1,n_2),&\hbox{in}\;\Omega_T,\\ \\
     		U\cdot\nabla w_1-\Delta w_1+\alpha_1\, w_1=\beta_1\,n_2,&\hbox{in}\;\Omega_T,\\ \\
     		U\cdot\nabla w_2-\Delta w_2+\alpha_2\, w_2=\beta_2\,n_1,&\hbox{in}\;\Omega_T,\\ \\
     		\partial_t U -\nu \Delta U+(U\cdot \nabla)U+ \nabla p+Q(n_1,n_2) \nabla\phi = \bzeros,  \; \;  \vdiv U=0,&\hbox{in}\;\Omega_T,\\
     		{}\\
     		n_i(t=0, x)=n_{i0}( x), \; U(t=0,\bx)=U_{0}( x),&\hbox{in}\;\Omega,\\
     		{}\\
     		U=\bzeros \quad \mbox{ and } \Big(d_i(n_i)\nabla n_i+ \chi_i(n_i)\nabla( w_i)\Big) \eta=0,&\hbox{on}\;\Sigma_T.
     	\end{array}
     	\right.
     \end{equation}

     \noindent In the proof of the existence of weak solutions, we will use the following assumptions.
     
     \noindent We assume that for $i \in\{1,2\}$, the function $d_{n_i}: \R \to \R^+ $ is continuous and satisfying the following: 
     \begin{equation}\label{assump:diffusion}
     	\underline{d_i} \leq d_{n_i}(r)\leq \bar{d_i}\quad  \forall r\in \R \;\mbox{and}  \quad \forall i\in\{1,2\} \; 
     \end{equation}
     where $\underline{d_i}$ and $\bar{d_i}$ are strictly positive constants.
     
     For the reaction terms ${F}_{i}$, they are continuous functions and there exists a constant $C_F$ such that
     \begin{equation}\label{Est:seconde}
     	\forall n_1,n_2\geq 0, \quad  {F}_{1}(0,n_2) \geq 0, \quad {F}_{2}(n_1,0) \geq 0 \;\mbox{and}
     	\quad \sum_{i=1}^2{F}_{i}(n_1,n_2)\,n_i\leq  C_F  (1+ n_1^2+ n_2^2). 
     \end{equation}
     
     \noindent  Regarding the function $Q$, we assume it is a continuous function and there exists constant $C_Q>0$ such that
     \begin{equation}\label{assump-G}
     	\abs{Q(n_1,n_2)}\leq C_Q(1+\abs{n_1}+\abs{n_2}) \text{ for all $n_1,n_2\in \R$}. 
     \end{equation}  
     
     \noindent Moreover, we assume that $$ \nabla\phi\in \big( \L^{d+2}(\Omega) \big)^d \quad \mbox{ and } \phi \mbox{  is independent of time}$$ stands for the gravitational potential produced by the action of physical forces on the species.

     \noindent Finally, we assume that initial conditions are 
     \begin{equation}\label{assump-initial}
     	n_{i,0}\geq 0,\quad n_{i,0}\in \L^2(\Omega),\quad U_{0}\in \bH.
     \end{equation}
     Now we define what we mean by weak solution of
     the system \eqref{CrossFF}. We also supply our main existence result.
     \begin{Definition}\label{defSol} 
     	We say that  $(n_1,n_2,w_1,w_2,U)$ is a weak solution to problem (\ref{CrossDiff}), if  $n_i$ is nonnegative,
     	\begin{equation*}\begin{split}
     			&n_i \in L^\infty(\Om_T) \cap L^2(0,T; \H^1(\Omega))\cap C(0,T;L^2(\Om)),\;\;\partial_t n_i\in\L^2(0,T;(H^1(\Om))^\prime),\\
     			&w_i\in  L^\infty(0,T;W^{2,p}(\Omega))\,\, \text{ for all $p>1$},\\
     			&U \in \L^2(0,T; \bV) \cap C\big([0,\:T]; \: \bH \big),\;\;\partial_t U\in\L^1(0,T;\bV^\prime),
     	\end{split}\end{equation*}
     	and the following identities hold 
     	\begin{equation}\label{wf-1}
     		\begin{split}
     			&\displaystyle \int_0^{T}\left \langle\partial_t n_i, \psi_i \right\rangle_{(H^1)^\prime,H^1}\,dt
     			-\iint_{\Omega_T}U\cdot\nabla n_i \,\psi_i \,dx\,dt + \iint_{\Omega_T} d_i({ n_i} )\nabla n_i\cdot\nabla \psi_i \,dx\,dt  \\ 
     			&\qquad \qquad \qquad + \iint_{\Omega_T}  \chi_i(n_i)\nabla w_i\cdot\nabla \psi_i \,dx\,dt 
     			= \iint_{\Omega_T} F_i( n_1,n_2)\psi_i \,dx\,dt,\\
     			&		\displaystyle -\iint_{\Omega_T}U\cdot\nabla w_1 \varphi_1 \,dx\,dt + \iint_{\Omega_T} \nabla w_1 \nabla  \varphi_1 \,dx\,dt   
     			= \iint_{\Omega_T} (\beta_1n_2-\alpha_1w_1)\varphi_1 \,dx\,dt,\\
     			&		\displaystyle -\iint_{\Omega_T}U\cdot\nabla w_2 \varphi_2 \,dx\,dt + \iint_{\Omega_T} \nabla w_2 \nabla  \varphi_2 \,dx\,dt   
     			= \iint_{\Omega_T} (\beta_2n_1-\alpha_2w_2)\varphi_2 \,dx\,dt,\\
     			&		\displaystyle	 \int_0^{T} \left\langle\partial_t U,\Psi\right\rangle_{\bV^\prime,\bV} \,dt   +  \nu \int_{\Omega}  \nabla U : \nabla\Psi \,dx\,dt    + \iint_{\Omega_T} (U \cdot \nabla) U\cdot \Psi \,dx\,dt \\
     			&\qquad \qquad \qquad + \iint_{\Omega_T} Q( n_1,n_2) \nabla\phi \cdot\Psi \,dx\,dt =\bzeros,
     	\end{split}\end{equation}
     	for all test functions $\psi_i,\;\varphi_i \in  L^2(0,T; \H^1(\Omega))$ 
     	and $\Psi\in \L^2(0,T; \bV) $, for $i=1,2$. 
     \end{Definition}
     \begin{Theorem}\label{theo-weak}
     	Assume that conditions \eqref{Est:seconde} and (\ref{assump-initial}) hold. If $n_{i,0}\in L^\infty(\Omega)$
     	with $0\le n_{i,0}\le u_{i,m}$ a.e. in $\Omega$ for $i=1,2$, then the problem (\ref{CrossFF}) has a weak solution in the sense of Definition \ref{defSol}.
     \end{Theorem}

     Our proof is based on approximation systems to which we can apply
     the Schauder fixed-point theorem to prove the convergence to weak
     solutions of the approximations. Let us now put our own contributions into a
     perspective. Our proof is based on introducing the following system for $i=1,2$
     \begin{equation}\label{S-reg}
     	\left\{
     	\begin{array}{ll}
     		\displaystyle
     		\partial_t n_i+ U\cdot\nabla n_i-  \vdiv\Big(d_i({n_i} ) \nabla n_i+\chi_{i,\eps}({n}_i)\nabla w_i \Big)= F_{i,\eps}({n}_1,{n}_2),&\hbox{in}\;\Omega_T,\\ \\
     		U\cdot\nabla w_1-\Delta w_1+\alpha_1w_1=\beta_1\,\overline{n}_2,&\hbox{in}\;\Omega_T,\\ \\
     		U\cdot\nabla w_2-\Delta w_2+\alpha_2w_2=\beta_2\,\overline{n}_1,&\hbox{in}\;\Omega_T,\\ \\
     		\partial_t U -\nu \Delta U+(U\cdot \nabla)U+ \nabla p+Q(\overline{n}_1,\overline{n}_2) \nabla\phi = \bzeros,  \; \;  \vdiv U=0,&\hbox{in}\;\Omega_T,\\
     		{}\\
     		n_i(t=0, x)=n_{i0}( x), \; U(t=0,\bx)=U_{0}( x),&\hbox{in}\;\Omega,\\
     		{}\\
     		U=\bzeros,  \quad\nabla w_i\cdot \eta \, \mbox{ and } \,\Big(d_i(n_i)\nabla n_i+ \chi_i(n_i)\nabla( w_i)\Big) \eta=0,&\hbox{on}\;\Sigma_T,
     	\end{array}
     	\right.
     \end{equation}
     for each fixed $\varepsilon >0$, where $ \overline{n}_i$
     is a fixed function. Herein
     $$
     F_{i,\eps}(r_1,r_2)=\frac{F_{i}(r_1,r_2)}{1+\varepsilon\abs{F_{i}(r_1,r_2)}}\quad \text{and}
     \quad\chi_{i,\eps}(r)=\frac{\chi_i(r)}{1+\varepsilon\abs{\chi_i(r)}},\qquad \mbox{for a.e. $r, r_1,r_2 \in {\R}$}.
     $$
     To prove Theorem \ref{theo-weak} we first
     prove existence of solutions to the problem (\ref {S-reg}) by
     applying the Schauder fixed-point theorem (in an appropriate
     functional setting), deriving a priori estimates, and then passing
     to the limit in the approximate solutions using monotonicity and
     compactness arguments. Having proved existence to the system
     (\ref{S-reg}), the goal is to send the regularization parameter
     $\varepsilon$ to zero in sequences of such solutions to fabricate weak
     solutions of the original systems (\ref{CrossFF}). Again
     convergence is achieved by a priori estimates and compactness
     arguments.
     
     \subsection{The fixed-point method}\label{Sect:nondegenerate}
     In this section we prove, for each fixed $\varepsilon> 0$, the existence of
     solutions to the fixed problem (\ref {S-reg}), by applying the
     Schauder fixed-point theorem.
     
     For technical reasons, we need to extend the function
     $F_{i,\eps}$ so that it becomes defined for all $(r_1,r_2)\in
     {\R}\times{\R}$. We do this by setting
     \begin{eqnarray}
     	F_{i,\eps}(r_1,r_2)= \left\{\begin{array}{ll}
     		F_{i,\eps}(r_1,0),  &\mbox{if $r_1\geq 0$, $r_2<0$},\\
     		F_{i,\eps}(0,r_2),  &\mbox{if $r_1< 0$, $r_2\geq 0$},\\
     		F_{i,\eps}(0,0),    &\mbox{if $r_1< 0$, $r_2< 0$}.
     	\end{array}
     	\right.
     \end{eqnarray}
     
     Since we use Schauder fixed-point theorem, we need to introduce the
     following closed subset of the Banach space $L^2(\Om_T)$:
     \begin{equation}
     	\A=\{(n_1,n_2)\in L^2(\Om_T;\R^2): 0\le n_1(t,x),n_2(t,x)\le M,
     	\,\mbox{for a.e.}\,(t,x)\in \Om_T \},
     \end{equation}
     where $M$ is a positive constant to be fixed in Lemma \ref{lem-classic:est} below.
     
     \subsection{Existence result to the fixed problem}
     In this section, we omit the dependence of the solutions on the
     parameter $\varepsilon$. With $ (\overline{n}_1, \overline{n}_2)\in \A$ fixed, let $w_i$ and $U$ be the unique solutions of the system
     \begin{equation}\label{S1-v}
     	\begin{cases}
     		\partial_t U -\nu \Delta U+(U\cdot \nabla)U+ \nabla p+Q(\overline{n}_1,\overline{n}_2) \nabla\phi = \bzeros,  \; \;  \vdiv U=0,&\hbox{in}\;\Omega_T,\\
     		\displaystyle U\cdot\nabla w_1-\Delta w_1+\alpha_1\,w_1=\beta_1\,\overline{n}_2,&\hbox{in}\;\Omega_T,\\ 
     		\displaystyle U\cdot\nabla w_2-\Delta w_2+\alpha_2\,w_2=\beta_2\,\overline{n}_1,&\hbox{in}\;\Omega_T,\\ 
     		U(t=0,\bx)=U_{0}( x),&\hbox{in}\;\Omega,\\
     		U=\bzeros, \quad \nabla w_i\cdot \eta=0,&\hbox{on}\;\Sigma_T,
     	\end{cases}
     \end{equation}
     for $i=1,2$. Given the functions $w_i$ and $U$, let $n_i$ be the unique
     solution of the quasilinear parabolic problem
     \begin{equation}\label{S1-u}
     	\left\{\begin{array}{ll}
     		\displaystyle
     		\partial_t n_i+ U\cdot\nabla n_i-  \vdiv\Big(d_i({n_i} ) \nabla n_i+\chi_{i,\eps}({n}_i)\nabla w_i \Big)= F_{i,\eps}({n}_1,{n}_2),  \; \;  \vdiv U=0,&\hbox{in}\;\Omega_T,\\ 
     		n_i(t=0, x)=n_{i,0}( x),&\hbox{in}\;\Omega,\\
     		\Big(d_i(n_i)\nabla n_i+ \chi_{i,\eps}(n_i)\nabla( w_i)\Big) \cdot \eta=0,&\hbox{on}\;\Sigma_T,
     	\end{array}
     	\right.\end{equation}
     for $i=1,2$. In (\ref{S1-v})- (\ref{S1-u}), $U_0$ and
     $n_{i,0}$ are functions satisfying the hypothesis of Theorem \ref{theo-weak} for $i=1,2$.
     
     Observe that for any fixed $  (\overline{n}_1, \overline{n}_2) \in \A$, problem (\ref{S1-v})
     is a pure Navier-Stokes equation coupled weakly to a an elliptic equation for $w_i$ for $i=1,2$,
     so we have immediately the following lemma (see for e.g. \cite{Ladyzenskaia:1968}).
     
     \begin{Lemma}
     	If  $U_0\in \bH$, then the system (\ref {S1-v}) has a
     	unique solution $(U, w_i)\in \L^2(0,T; \bV) \times L^\infty(0,T;W^{2,p}(\Omega))$ for $i=1,2$, for all $p>1$.
     \end{Lemma}
     
     \noindent We have the following lemma for the quasilinear problem (\ref{S1-u}):
     \begin{Lemma}
     	If  $n_{i,0} \in L^\infty(\Omega)$, then, for any $\varepsilon >0$,
     	there exists a unique weak solution $n_{i} \in L^\infty (\Om_T)\cap
     	L^2(0,T;H^1(\Omega))$ to problem (\ref{S1-u}) for $i=1,2$. \label{lem2.2}
     \end{Lemma}

     \subsection{The fixed-point method}
     \label{Sect:fixed}
     In this subsection, we introduce a map $\Gamma:\A\to \A$ such that
     $\Gamma( \overline{n}_1,\overline{n}_2)=(n_1,n_2)$, where $(n_1,n_2)$ solves (\ref {S1-u}),
     \textit{i.e.}, $\Gamma$ is the solution operator of (\ref {S1-u})
     associated with the coefficient $ (\overline{n}_1,\overline{n}_2)$ and the solution $w_i$ and $U$
     coming from (\ref {S1-v}) for $i=1,2$. By using the Schauder fixed-point
     theorem, we prove that the map $\Gamma$ has a fixed point for (\ref{S1-v})-(\ref {S1-u}).
     
     First, let us show that $\Gamma$ is a continuous mapping. For this,
     we let $(\overline{n}_{1, \kappa},\overline{n}_{2,\kappa})_\kappa $ be a sequence in $\A$ and
     $(\overline{n}_1,\overline{n}_2) \in \A$ be such that $(\overline{n}_{1,\kappa},\overline{n}_{2,\kappa})\to  (\overline{n}_1,\overline{n}_2)$ in $L^2(\Om_T;\R^2)$ as $\kappa \to \infty$. Define
     $({n}_{1,\kappa},{n}_{2,\kappa})=\Gamma(\overline{n}_{1, \kappa},\overline{n}_{2,\kappa})$, \textit{i.e.},
     ${n}_{1,\kappa},{n}_{2,\kappa}$ is the solution of (\ref{S1-u}) associated with
     $ (\overline{n}_{1, \kappa},\overline{n}_{2,\kappa})$ and the solutions $w_{i,\kappa}$ and $U_\kappa$ of
     (\ref {S1-v}) for $i=1,2$. The goal is to show that $({n}_{1, \kappa},{n}_{2,\kappa})$ converges to $\Gamma(\overline{n}_{1},\overline{n}_{2})$ in $L^2(\Om_T)$. \\
     We start with the following lemma where the proof can be found in (\cite{Ladyzenskaia:1968} and in \cite{Te01} Lemma 4.3) so we omit it
     \begin{Lemma}\label{lem-reg1}
     	If  $U_0\in \bH$, then the solution $({U}_{ \kappa},{w}_{i,\kappa})_\kappa$  to the system (\ref {S1-v}) is uniformly bound in $\L^2(0,T; \bV) \times L^\infty(0,T;W^{2,p}(\Omega))$ for $i=1,2$, for all $p>1$. Moreover,
     	$\partial_t {U}_{ \kappa}$ is uniformly bounded in $\L^1(0,T; \bV^\prime)$.
     \end{Lemma}
     
     \begin{Lemma}\label{lem-classic:est}
     	The solution $({n}_{1, \kappa},{n}_{2,\kappa})_\kappa$ to problem (\ref {S1-u}) satisfies
     	\begin{itemize}
     		\item[(i)]\quad There exists a constant $M> 0$ such that
     		$$
     		0\le{n}_{1, \kappa}(t,x),{n}_{2,\kappa}(t,x)\le M \mbox{ for a.e. } (t,x)\in \Om_T.
     		$$
     		\item[(ii)]\quad The sequence $({n}_{1, \kappa},{n}_{2,\kappa})_\kappa$ is bounded in
     		$L^2(0,T;H^1(\Omega,R^2))\cap L^\infty(0,T;L^2(\Omega,\R^2))$.
     		\item[(iii)]\quad The sequence $(U_\kappa)_\kappa$ is bounded in
     		$L^2(0,T;H^1(\Omega,R^3))\cap L^\infty(0,T;L^2(\Omega,\R^3))$.
     		\item[(iv)]\quad The
     		sequence $({n}_{1, \kappa},{n}_{2,\kappa})_\kappa$ is relatively compact in $L^2(\Om_T,\R^2)$.
     		\item[(v)]\quad The sequence $(U_{\kappa})_\kappa$ is relatively compact in $\L^2(\Om_T)$.
     	\end{itemize}
     \end{Lemma}
     
     \begin{proof}
     	$(i)$ {\bf Nonnegativity.}
     	Multiplying (\ref {S1-u}) by $-n_{i,\kappa}^-=\displaystyle\frac{n_{i,\kappa}-\abs{n_{i,\kappa}}}{2}$ and
     	integrating over $\Omega$, we get
     	\begin{eqnarray}\begin{array}{ll}\displaystyle
     			&\displaystyle\frac{1}{2}\frac{d}{\,dt}\int_\Omega \abs{n_{i,\kappa}^-}^2  \,dx
     			- \int_\Omega U_\kappa \cdot\nabla n_{i,\kappa} \, n_{i,\kappa}^-  \,dx
     			+\underline{d}_i\int_\Omega \abs{\nabla n_{i,\kappa}^-}^2 \,dx\\
     			& \displaystyle\qquad \qquad \qquad \qquad=\int_\Omega \chi_{\eps}(n_{i,\kappa})\nabla w_{i,n} \cdot \nabla n_{i,\kappa}^-  \,dx-\int_\Omega F_{i,\eps}({n}_{1,\kappa},{n}_{2,\kappa}) n_{i,\kappa}^-  \,dx,
     	\end{array}\end{eqnarray}
     	for $i=1,2$. Recall that $\vdiv U_\kappa=0$ in $\Om_T$ and $U_\kappa=0$ on $\Sigma_T$, so we have for $i=1,2$
     	$$
     	\int_\Omega U_\kappa \cdot\nabla n_{i,\kappa}\, n_{i,\kappa}^-  \,dx
     	=\frac{1}{2}\int_\Omega U_\kappa \cdot \nabla  \abs{n_{i,\kappa}}^2\,dx
     	=-\frac{1}{2}\int_\Omega \vdiv U_\kappa \, \abs{n_{i,\kappa}}^2\,dx
     	+\frac{1}{2}\int_{\partial \Omega}\abs{n_{i,\kappa}}^2  U_\kappa \cdot \eta d\sigma=0.
     	$$
     	Using this and since $\chi_\eps(s)=0$, $F_{i,\eps}(s_1,s_2)=0$ for $s,s_1\le 0$, $s_2\in
     	{\R}$, and according to the positivity of the third term of the
     	left-hand side, we obtain
     	$$
     	\frac{1}{2}\frac{d}{\,dt}\int_\Omega \abs{n_{i,\kappa}^-}^2\,dx\le 0\qquad\text{for $i=1,2$.}
     	$$
     	Since the data $n_{i,0}$ is is nonnegative, we deduce that $n_{i,\kappa}^-=0$ for $i=1,2$.\\
     	
     	\noindent {\bf Boundedness in $L^1$ and $L^\infty$.}
     	To obtain the $L^1$ bound of $n_{i,\kappa}$ for $i=1,2$, we integrate the equation (\ref{S1-u}) over $\Omega$, to deduce
     	\begin{equation}\begin{split}\label{L1-ineq1}
     			\displaystyle\frac{d}{\,dt}\sum_{i=1,2}\int_\Omega n_{i,\kappa}  \,dx&=\sum_{i=1,2}\int_\Omega F_{i,\eps}({n}_{1,\kappa},{n}_{2,\kappa})   \,dx\\
     			&\leq \sum_{i=1,2} \int_\Omega {n}_{i,\kappa}(a_i-b_i\,{n}_{1,\kappa}-c_i\,{n}_{2,\kappa})   \,dx\\
     			&\leq \sum_{i=1,2} a_i \int_\Omega {n}_{i,\kappa} \,dx\\
     			&\leq \max\{a_1,a_2\}\sum_{i=1,2} \int_\Omega {n}_{i,\kappa} \,dx
     	\end{split}\end{equation}
     	where we have used the nonnegativity of $n_{i,\kappa}$ and
     	$$
     	\int_\Omega U_\kappa \cdot \nabla  n_{i,\kappa}  \,dx=-\int_\Omega n_{1,\kappa}\,\Div U_\kappa  \,dx=0,
     	$$
     	for $i=1,2$. An application of Gr\"{o}nwall inequality to \eqref{L1-ineq1}, we obtain for $i=1,2$
     	\begin{equation}\begin{split}\label{L1-ineq2}
     			\norm{{n}_{i,\kappa}}_{L^\infty(0,T;L^1(\Om))} \leq C,
     	\end{split}\end{equation}
     	for some constant $C>0$.\\
     	
     	\noindent In the next step we prove $L^\infty$ bound of $n_{i,\kappa}$ for $i=1,2$.
     	We multiply (\ref{S1-u}) for $i=1$ by $(n_{i,\kappa})^{p-1}$ and integrate over $\Omega$. The result is
     	\begin{equation}
     		\begin{split}\label{eq2:est-class}
     			&\displaystyle\frac{1}{p}\frac{d}{\,dt}\int_\Omega
     			\abs{n_{1,\kappa}}^p  \,dx
     			+(p-1)\underline{d}_1\int_\Omega (n_{1,\kappa})^{p-2}\abs{\nabla  n_{1,\kappa}}^{2}\,\,dx
     			+\frac{p-1}{p+m_1-1}\int_\Omega \nabla  w_{1,\kappa}\cdot \nabla  (n_{1,\kappa})^{p+m_1-1}\,\,dx\\
     			&\qquad \qquad \qquad \qquad +b_1 \int_\Omega (n_{1,\kappa})^{p+1}\,\,dx\\
     			&\quad \leq \int_\Omega \pt n_{1,\kappa} \,(n_{1,\kappa})^{p-1}  \,dx
     			+\int_\Omega U_\kappa \cdot \nabla  n_{1,\kappa}\, (n_{1,\kappa})^{p-1}  \,dx
     			+\int_\Omega d_1(n_{1,\kappa})\nabla  n_{1,\kappa}\cdot \nabla   (n_{1,\kappa})^{p-2} \,\,dx\\
     			&\qquad \qquad \qquad \qquad 
     			+\int_\Omega \chi_{1,\eps}(n_{1,\kappa})\nabla  w_{1,\kappa}\cdot \nabla  (n_{1,\kappa})^{p-1}\,\,dx
     			++b_1 \int_\Omega (n_{1,\kappa})^{p+1}\,\,dx\\
     			&\quad =\int_\Omega F_{1,\eps}({n}_{1,\kappa},{n}_{2,\kappa})  (n_{1,\kappa})^{p-1}\,\,dx
     			+b_1 \int_\Omega (n_{1,\kappa})^{p+1}\,\,dx\leq a_1 \int_\Omega (n_{1,\kappa})^{p}\,\,dx.
     	\end{split}\end{equation}
     	Herein, we have used
     	$$
     	\int_\Omega U_\kappa\cdot \nabla  n_{1,\kappa}\, (n_{1,\kappa})^{p-1}  \,dx
     	=\frac{1}{p}\int_\Omega U_\kappa\cdot \nabla  (n_{1,\kappa})^{p}  \,dx
     	=-\frac{1}{p}\int_\Omega \Div \, U_\kappa\,(n_{1,\kappa})^{p}  \,dx=0.
     	$$
     	We observe that
     	\begin{equation}\label{esti-Gagl1}
     		\int_\Omega (n_{1,\kappa})^{p-2}\abs{\nabla n_{1,\kappa}}^{2}\,\,dx
     		=\frac{4(p-1)}{p^2}\int_\Omega \abs{\nabla (n_{1,\kappa})^{\frac{p}{2}}}^{2}\,dx.
     	\end{equation}
     	Moreover, from the equation of $w_{1,\kappa}$ in (\ref{S1-u}) and $\Div\, U_\kappa=0$ in $\Om_T$, we deduce
     	\begin{equation}
     		\begin{split}\label{esti-Gagl2}
     			\int_\Omega \nabla w_{1,\kappa}\cdot \nabla (n_{1,\kappa})^{p+m_1-1}\,\,dx
     			&=-\alpha_1\int_\Omega {w}_{1,\kappa}(n_{1,\kappa})^{p+m_1-1}\,\,dx
     			+\beta_1\int_\Omega \overline{n}_{1,\kappa}\,(n_{1,\kappa})^{p+m_1-1}\,\,dx\\
     			&\geq -\alpha_1\int_\Omega {w}_{1,\kappa}(n_{1,\kappa})^{p+m_1-1}\,\,dx.
     	\end{split}\end{equation}
     	Now, we use \eqref{esti-Gagl1}-\eqref{esti-Gagl2} and Young inequality to deduce from \eqref{eq2:est-class}
     	(recall that $1 \leq m_i<2$ for $i=1,2$)
     	\begin{equation}
     		\begin{split}\label{esti-Gagl3}
     			&\displaystyle\frac{d}{\,dt}\int_\Omega\abs{n_{1,\kappa}}^p  \,dx
     			+\frac{4(p-1)}{p}\int_\Omega \abs{\nabla (n_{1,\kappa})^{\frac{p}{2}}}^{2}\,dx
     			+b_1 \int_\Omega (n_{1,\kappa})^{p+1}\,\,dx\\
     			&\qquad \qquad \qquad \leq a_1 p \int_\Omega (n_{1,\kappa})^{p}\,\,dx
     			+\frac{\alpha_1p(p-1)}{p+m_1-1}\int_\Omega {w}_{1,\kappa}(n_{1,\kappa})^{p+m_1-1}\,\,dx\\
     			&\qquad \qquad \qquad \leq  C(a_1,p,m_1)\Bigl(1+\theta/2 \int_\Omega\abs{n_{1,\kappa}}^{p+1}  \,dx
     			+\theta/2 \int_\Omega\abs{n_{1,\kappa}}^{p+1}  \,dx+\int_\Omega\abs{w_{1,\kappa}}^{\frac{p+1}{2-m_1}}  \,dx\Bigl)\\
     			&\qquad \qquad \qquad \leq  C(a_1,p,m_1)\Bigl(1+\theta \int_\Omega\abs{n_{1,\kappa}}^{p+1}  \,dx
     			+\int_\Omega\abs{w_{1,\kappa}}^{p+1}  \,dx\Bigl),
     	\end{split}\end{equation}
     	where $C(a_1,p,m_1)>0$ is a constant depending on $a_1,\,p$ and $m_1$.
     	An application of Gagliardo-Nirenberg-Sobolev inequality and Young inequality to
     	$\displaystyle\int_\Omega\abs{n_{1,\kappa}}^{p+1}  \,dx$ and
     	$\displaystyle \int_\Omega\abs{n_{1,\kappa}}^{p}  \,dx$, respectively, we get from \eqref{esti-Gagl3} and \eqref{L1-ineq2}
     	\begin{equation}
     		\begin{split}\label{esti-Gagl4}
     			&\displaystyle\frac{d}{\,dt}\int_\Omega\abs{n_{1,\kappa}}^p  \,dx
     			+\frac{4(p-1)}{p}\int_\Omega \abs{\nabla (n_{1,\kappa})^{\frac{p}{2}}}^{2}\,dx
     			+b_1 \int_\Omega (n_{1,\kappa})^{p+1}\,\,dx\\
     			&\qquad  \leq  C(a_1,p,m_1,\abs{\Om})\Biggl(1+\theta \int_\Omega\abs{n_{1,\kappa}}\, dx
     			\times \Bigl(\int_\Omega\abs{n_{1,\kappa}}^{p}  \,dx
     			+\int_\Omega \abs{\nabla (n_{1,\kappa})^{\frac{p}{2}}}^{2}\,dx\Bigl)
     			+\int_\Omega\abs{w_{1,\kappa}}^{p+1}  \,dx\Biggl)\\
     			&\qquad  \leq \tilde{C}(a_1,p,m_1,\abs{\Om})\Biggl(1+\theta \int_\Omega\abs{n_{1,\kappa}}^{p}  \,dx
     			+\theta \int_\Omega \abs{\nabla (n_{1,\kappa})^{\frac{p}{2}}}^{2}\,dx
     			+\int_\Omega\abs{w_{1,\kappa}}^{p+1}  \,dx\Biggl),
     	\end{split}\end{equation}
     	for some constants $C(a_1,p,m_1,\abs{\Om}), \tilde{C}(a_1,p,m_1,\abs{\Om})>0$ depending on
     	$a_1,\,p,\,m_1$ and $\abs{\Om}$. We choose $\theta$ sufficiently small to deduce from \eqref{esti-Gagl4}
     	\begin{equation}
     		\begin{split}\label{esti-Gagl4b}
     			&\displaystyle\frac{d}{\,dt}\int_\Omega\abs{n_{1,\kappa}}^p  \,dx
     			+C \int_\Omega (n_{1,\kappa})^{p+1}\,\,dx \leq \tilde{C}(a_1,p,m_1,\abs{\Om})
     			\Bigl(1+\int_\Omega\abs{w_{1,\kappa}}^{p+1}  \,dx\Bigl),
     	\end{split}\end{equation}
     	for some constant $C>0$. To control the integral in the right-side, we multiply the equation of $w_{1,\kappa}$ in (\ref{S1-u}) by $(w_{1,\kappa})^{p-1}$, we use $\Div \, U_\kappa=0$ in $\Om_T$ and Gagliardo-Nirenberg-Sobolev inequality to get
     	\begin{equation}
     		\begin{split}\label{esti-Gagl5}
     			\int_\Omega (w_{1,\kappa})^{p+1}\,\,dx
     			& \leq C(p,\abs{\Om})
     			\Biggl(\int_\Omega\abs{n_{1,\kappa}}\, dx \times \Bigl(\int_\Omega\abs{n_{1,\kappa}}^{p}  \,dx
     			+\int_\Omega \abs{\nabla (n_{1,\kappa})^{\frac{p}{2}}}^{2}\,dx\Bigl)\Biggl)\\
     			&\leq C(p,\abs{\Om},\beta_1)\Bigl(\int_\Omega\abs{n_{1,\kappa}}^{p}  \,dx
     			+\int_\Omega \overline{n}_{1,\kappa}\,({w}_{1,\kappa})^{p-1}\,\,dx\Bigl)\\
     			&\leq \tilde{C}(p,\abs{\Om},\beta_1)\Biggl(\int_\Omega\abs{n_{1,\kappa}}^{p}  \,dx
     			+\int_\Omega ({w}_{1,\kappa})^{p}\,\,dx\Biggl)\\
     			&\leq \hat{C}(p,\abs{\Om},\beta_1)\Biggl(1+\int_\Omega\abs{n_{1,\kappa}}^{p}  \,dx
     			+\theta\int_\Omega ({w}_{1,\kappa})^{p}\,\,dx\Biggl),
     	\end{split}\end{equation}
     	for some constants $C,\tilde{C},\hat{C}>0$. Again we $\theta$ sufficiently small to obtain from \eqref{esti-Gagl5}
     	\begin{equation}
     		\label{esti-Gagl6}
     		\int_\Omega (w_{1,\kappa})^{p+1}\,\,dx
     		\leq {C}(p,\abs{\Om},\beta_1)\Bigl(1+\int_\Omega\abs{n_{1,\kappa}}^{p}  \,dx\Bigl),
     	\end{equation}
     	for some constant $C>0$.\\
     	Observe that from \eqref{esti-Gagl4b} and \eqref{esti-Gagl6}, we deduce
     	\begin{equation}\label{esti-Gagl7}
     		\displaystyle\frac{d}{\,dt}\int_\Omega\abs{n_{1,\kappa}}^p  \,dx
     		+C \int_\Omega (n_{1,\kappa})^{p+1}\,\,dx \leq {C}(a_1,p,m_1,,\beta_1,\abs{\Om})
     		\Bigl(1+\int_\Omega\abs{n_{1,\kappa}}^{p}  \,dx\Bigl),
     	\end{equation}
     	for some constant $C>0$. Therefore an application of Gr\"onwall inequality, we arrive to
     	\begin{equation}\label{esti-Gagl7b}
     		\norm{n_{1,\kappa}}_{L^p(\Om)} \leq C(a_1,p,m_1,\beta_1,\abs{\Om})\quad \text{for all $t\in (0,T)$,}
     	\end{equation}
     	for some constant $C>0$. The consequence of \eqref{esti-Gagl7b} and the well-known
     	Moser–Alikakos iteration procedure (see for e.g. \cite{Alik79}) is the uniform $L^\infty$-bound
     	\begin{equation}\label{esti-Gagl8}
     		\norm{n_{1,\kappa}}_{L^\infty(\Om)} \leq C(a_1,p,m_1,,\beta_1,\abs{\Om})\quad \text{for all $t\in (0,T)$,}
     	\end{equation}
     	for some constant $C>0$.
     	
     	$(ii)$ We multiply the equation (\ref {S1-u}) by $n_{i,\kappa}$ and integrate
     	over $\Omega$ to obtain
     	\begin{eqnarray}\label{eq1:2est-class}
     		\begin{array}{ll}\displaystyle
     			\frac{1}{2}\frac{d}{\,dt}\int_\Omega \abs{n_{i,\kappa}}^2  \,dx 
     			& \displaystyle+ \int_\Omega U_\kappa \cdot\nabla n_{i,\kappa} \, n_{i,\kappa}  \,dx
     			+\displaystyle \underline{d}_i\int_\Omega
     			\abs{\nabla n_{i,\kappa}}^2  \,dx \\\\
     			&\displaystyle =\int_\Omega \chi_{\eps}(n_{i,\kappa})\nabla w_{i,n} \cdot \nabla n_{i,\kappa}  \,dx-\int_\Omega F_{i,\eps}({n}_{1,\kappa},{n}_{2,\kappa}) n_{i,\kappa} \,dx.
     	\end{array}\end{eqnarray}
     	Exploiting the boundedness of $n_{i,\kappa}$and $\Div \, U_\kappa=0$ in $\Om_T$, we get  
     	$$
     	\int_\Omega U_\kappa \cdot\nabla n_{i,\kappa} \, n_{i,\kappa}  \,dx=0,
     	$$
     	and that the second and the third integrals of the right-hand side are bounded
     	independently of $\kappa$, for $i=1,2$ . Then by Young inequality
     	\begin{equation}
     		\displaystyle\frac{1}{2}\frac{d}{\,dt}\int_\Omega \abs{n_{i,\kappa}}^2  \,dx +C_2
     		\int_\Omega \abs{\nabla n_{i,\kappa}}^2  \,dx \le C_3,
     	\end{equation}
     	for some constants $C_2,C_3>0$ independent of $\kappa$. This completes
     	the proof of $(ii)$.
     	
     	$(iii)$ In this step, we multiply the equation (\ref {S1-v}) by $U_{\kappa}$ and integrate
     	over $\Omega$ to obtain
     	\begin{equation}\label{est-fluid1}
     		\begin{array}{l}
     			\displaystyle	\frac{1}{2} \int_\Omega |U_{\kappa}(\tau,x)|^2 \,dx +\nu \int_0^\tau   \int_\Omega | \nabla U_{\kappa}|^2 \,dx\,dt  + \int_0^\tau \int_{\Omega} (U_{\kappa}\cdot \nabla) U_{\kappa}\cdot U_{\kappa} \,dx \,dt \\
     			\displaystyle\hskip4cm + \int_0^\tau \int_{\Omega} Q\big(\overline{n}_{1,\kappa},\overline{n}_ {2,\kappa}\big) \nabla\phi \cdot U_{\kappa}  \,dx \,dt=\frac{1}{2} \int_\Omega |U_{0}(x)|^2 \,dx.
     		\end{array}
     	\end{equation}
     	Observe that, since $\vdiv U_{\kappa}=0$ and $U_{\kappa}=\bzeros$ on $\partial \Omega$, we get
     	$$ \displaystyle \int_\Omega  (U_{\kappa}\cdot\nabla) U_{\kappa} \cdot U_{\kappa} \,dx= \frac{1}{2}\int_\Omega\nabla (U_{\kappa})^2 U_{\kappa}\,dx=-\frac{1}{2}\int_\Omega\vdiv(U_{\kappa}) (U_{\kappa})^2 \,dx+\frac{1}{2}\int_{\partial\Omega} U_{\kappa}\,{(U_{\kappa})^2}^T\eta=\bzeros.$$
     	Using this to deduce from \eqref{est-fluid1}	
     	\begin{equation}\label{granwv}
     		\begin{array}{l}
     			\displaystyle  \frac{1}{2} \int_\Omega |U_{\kappa}(\tau, x)|^2 \,dx  + \displaystyle  \nu \int_0^\tau   \int_\Omega | \nabla U_{\kappa}|^2 \,dx\,dt       \displaystyle  \leq  \frac{1}{2} \int_\Omega |U_{0}(\bx)|^2 \,dx-  \int_0^\tau \int_{\Omega} Q\big(\overline{n}_{1,\kappa},\overline{n}_ {2,\kappa}\big) \nabla\phi \cdot U_{\kappa}  \,dx \,dt .
     		\end{array}
     	\end{equation}
     	Using  (\ref{assump-G}) and Young inequality, we have
     	$$\begin{array}{ll}
     		\displaystyle I:&=\displaystyle\Big| \int_0^\tau\int_{\Omega} Q\big(\overline{n}_{1,\kappa},\overline{n}_ {2,\kappa}\big) \nabla\phi \cdot U_{\kappa}  \,dx \,dt  \Big| \\
     		&\leq \displaystyle C_Q\Big(|\Omega_\tau|+ \int_0^\tau\int_{\Omega}\abs{ \nabla\phi \cdot U_{\kappa} } \,dx \,dt\Big)  \\
     		\displaystyle  &\leq\displaystyle C_Q\Big(|\Omega_\tau|+\frac{1}{2}\int_0^\tau\int_{\Omega}| \nabla\phi|^{d+2}\,dx\,dt+\frac{1}{2}\int_0^\tau\int_{\Omega}{|U_{\kappa}|}^2\,dx\,dt\Big).
     	\end{array}
     	$$
     	Using this and exploiting the assumption $ \nabla\phi\in (\L^{2}(\Omega))^d$ to deduce from (\ref{granwv})
     	\begin{equation} \begin{split}\label{granwvv}
     			\displaystyle  \frac{1}{2} \int_\Omega |U_{\kappa}(\tau,\bx)|^2 \,dx  + \displaystyle  \nu \int_0^\tau   \int_\Omega | \nabla U_{\kappa}|^2 \,dx\,dt       & \leq  \frac{C_Q}{2}\int_0^\tau\int_{\Omega}{|U_{\kappa} (t,\bx)|}^2\,dx\,dt
     			\\&\qquad+ \frac{1}{2} \int_\Omega |U_{0}(\bx)|^2 \,dx+   \tilde{C_Q}(T, \bu_0,|\Omega|,\phi).
     	\end{split}\end{equation}
     	An application of Gronwall's inequality, we obtain 
     	$$\displaystyle  \int_\Omega |U_{\kappa}(\tau,\bx)|^2 \,dx \le C\quad \text{for all }\tau \in(0,T),$$
     	for some constant $C>0$.
     	Consequently, we deduce from this and (\ref{granwvv})
     	\begin{equation}\label{EstComp}
     		\displaystyle  \max_{0<\tau<T}   \ \int_\Omega |U_{\kappa}(\tau,\bx)|^2 \,dx +\nu \int_0^T\int_{\Omega}  | \nabla U_{\kappa}|^2 \,dx\le C,
     	\end{equation}
     	for some constant $C>0$.
     	Therefore,  we deduce that $U_{\kappa}$ is uniformly bounded in $\displaystyle   \L^\infty\big(0,T; \bH \big)  \cap \L^2\Big(0,T; \bV\Big)$.
     	
     	$(iv)$ We multiply the equation (\ref {S1-u}) by $\varphi \in
     	L^2(0,T;H^1(\Omega))$ and we use the boundedness of $n_{i,\kappa}$ in $L^\infty$ and $U_\kappa$ in $L^2$,
     	the result is
     	\begin{eqnarray}
     		\begin{array}{ll}\displaystyle
     			\abs{\int_0^T\left \langle \partial_t n_{i,\kappa},\varphi\right \rangle \,dt}
     			&\displaystyle \leq \abs{\iint_{\Om_T} \varphi\, U_\kappa \cdot n_\kappa\,dx\,dt}
     			+\abs{\iint_{\Om_T} d_i(n_{i,\kappa})\, \nabla  n_{i,\kappa}\cdot \nabla  \varphi\,dx\,dt}\\
     			&\qquad \displaystyle +\abs{\iint_{\Om_T} \chi_{i,\eps}(n_{i,\kappa})\, \nabla  w_{i,\kappa}\cdot\nabla  \varphi\,dx\,dt}
     			+\abs{\iint_{\Om_T} F_{i,\eps}(n_{1,\kappa},n_{1,\kappa})\,  \varphi\,dx\,dt} \\
     			& \le C_4 \norm{U_\kappa}_{L^2(\Om_T)}\norm{\nabla  \varphi}_{L^2(\Om_T)}
     			+\overline{d} \norm{\nabla n_{i,\kappa}}_{L^2(\Om_T)} \left\|{\nabla \varphi}\right \|_{L^2(\Omega_T)}\\
     			&\qquad +\norm{\chi_{i,\eps}(n_{i,\kappa})}_{L^\infty(\Om_T)}
     			\left\|{\nabla w_{i,\kappa}}\right \|_{L^2(\Om_T)}
     			\left\|{\nabla \varphi}\right \|_{L^2(\Om_T)}\\
     			&\qquad \qquad \displaystyle + C_5\sum_{i=1,2}\norm{n_{i,\kappa}}_{L^2(\Om_T)} \norm{\varphi}_{L^2(\Om_T)}\\
     			& \le C_6\norm{\varphi}_{L^2(0,T;H^1(\Omega))},
     		\end{array}
     	\end{eqnarray}
     	for $i=1,2$ and for some constants $C_4,C_5,C_6>0$ independent of $\kappa$ and $\eps$.  We obtain the
     	bound
     	\begin{equation}
     		\left\|{ \partial_t n_{i,\kappa}}\right \|_{L^2(0,T;(H^1(\Omega))')} \le C.
     	\end{equation}
     	Then, $(iv)$ is a consequence of $(ii)$,  the uniform boundedness of $(\partial_t n_{i,\kappa} )_\kappa$ in 
     	$L^2(0,T;(H^1(\Omega)')$ and Aubin-Simon compactness theorem (see for e.g. \cite{Aubin-Simon}). \\
     	
     	$(v)$ Finally, using Lemma \ref{lem-reg1} and again Aubin-Simon compactness theorem, the space
     	$$
     	\Bigl\{U_\kappa \in L^2(0,T;\V);\,\,\partial_t U_\kappa\in L^2(0,T;\V^\prime)\Bigl\}\quad 
     	\text{is compactly embedded in }L^2(0,T;\H).
     	$$
     	This concludes the proof of Lemma \ref{lem-classic:est}.
     \end{proof}
     
     Now we have the following classical result (see \cite{Ladyzenskaia:1968}). 
     \begin{Lemma}\label{lem-classic:est-bis}
     	There exists a function $w_i \in L^2(0,T;H^1(\Omega))$ such that the
     	sequence $(w_{i,\kappa})_\kappa$ converges strongly to $w_i$ in
     	$L^2(0,T;H^1(\Omega))$ for $i=1,2$.
     \end{Lemma}
     
     Summarizing our findings so far, from Lemma \ref{lem2.2},
     \ref{lem-classic:est} and \ref{lem-classic:est-bis}, there exist
     functions $n_i,w_i,U \in L^2(0,T;H^1(\Omega))$ such that, up to extracting
     subsequences if necessary (for $i=1,2$),
     
     	\begin{itemize}
     		\item $n_{i,\kappa} \to n_i$ in $L^2(\Om_T)$ strongly, 
     		\item $w_{i,\kappa}$ $\to w_i $ in
$L^2(0,T;H^1(\Omega))$ strongly,
     		\item $U_{\kappa} \to U $ in $L^2(0,T; \H^1)$ strongly,
     	\end{itemize}
     and from this the continuity of $\Gamma$ on $\A$ follows.
     
     We observe that, from Lemma \ref{lem-classic:est}, $\Gamma(\A)$ is
     bounded in the set
     \begin{equation}\label{def-ww}
     	\mathcal{E}=\left\{ n_i \in L^2(0,T;H^1(\Omega)) \ : \ \partial_t n_i \in
     	L^2(0,T;(H^1(\Omega))') \right\},
     \end{equation}
     for $i=1,2$.
     By the results of \cite{Aubin-Simon}, $\mathcal{E} \hookrightarrow
     L^{2}(\Om_T)$ is compact, thus $\Gamma$ is compact. Now, by the
     Schauder fixed point theorem, the operator $\Gamma$ has a fixed
     point $n_{i,\eps}$ such that
     $\Gamma(n_{i,\eps})=n_{i,\eps}$ for $i=1,2$. Then there exists a solution
     ($n_{i,\eps}, w_{i,\eps},U_\varepsilon)$ of
     \begin{equation}\label{form:ueps}
     	\begin{split}
     		&\displaystyle \int_0^{T}\left \langle\partial_t n_{i,\eps}, \psi_i \right\rangle_{(H^1)^\prime,H^1}\,dt
     		-\iint_{\Omega_T}U_\eps\cdot\nabla n_{i,\eps} \,\psi_i \,dx\,dt + \iint_{\Omega_T} d_i({ n_{i,\eps}} )\nabla n_{i,\eps}\cdot\nabla \psi_i \,dx\,dt  \\ 
     		&\qquad \qquad \qquad \qquad + \iint_{\Omega_T}  \chi_{i,\eps}(n_{i,\eps})\nabla w_{i,\eps}\cdot\nabla \psi_i \,dx\,dt 
     		= \iint_{\Omega_T} F_{i,\eps}( n_{1,\eps},n_{2,\eps})\psi_i \,dx\,dt,\\
     		&		\displaystyle -\iint_{\Omega_T}U_\eps\cdot\nabla w_{1,\eps} \varphi_1 \,dx\,dt + \iint_{\Omega_T} \nabla w_{1,\eps} \nabla  \varphi_1 \,dx\,dt   
     		= \iint_{\Omega_T} (\beta_1n_{2,\eps}-\alpha_1w_{1,\eps})\varphi_1 \,dx\,dt,\\
     		&		\displaystyle -\iint_{\Omega_T}U_\eps\cdot\nabla w_{2,\eps} \varphi_2 \,dx\,dt + \iint_{\Omega_T} \nabla w_{2,\eps} \nabla  \varphi_2 \,dx\,dt   
     		= \iint_{\Omega_T} (\beta_2n_{1,\eps}-\alpha_2w_{2,\eps})\varphi_i \,dx\,dt,\\
     		&		\displaystyle	 \int_0^{T} \left\langle\partial_t U_\eps,\Psi\right\rangle_{\bV^\prime,\bV} \,dt   +  \nu \int_{\Omega}  \nabla U_\eps : \nabla\Psi \,dx\,dt    + \iint_{\Omega_T} (U_\eps \cdot \nabla) U_\eps\cdot \Psi \,dx\,dt \\
     		&\qquad \qquad \qquad \qquad+ \iint_{\Omega_T} Q( n_{1,\eps},n_{2,\eps}) \nabla\phi \cdot\Psi \,dx\,dt =\bzeros,
     \end{split}\end{equation}
     for all test functions $\psi_i,\;\varphi_i \in  L^2(0,T; \H^1(\Omega))$ 
     and $\Psi\in \L^2(0,T; \bV) $, for $i=1,2$. 
     
     \subsection{Existence of weak solutions}
     \label{Sect:proof-weak}
     We have shown in Section \ref{Sect:nondegenerate} that the problem
     (\ref {S-reg}) admits a solution $(n_{1,\varepsilon},n_{2,\varepsilon},w_{1,\varepsilon},w_{2,\varepsilon}, U_\eps)$. The goal in
     this section is to send the regularization parameter $\varepsilon$ to zero
     in sequences of such solutions to obtain weak solutions of the
     original system (\ref{CrossDiff}), (\ref {BC}) and (\ref {IC}). Note that, for each fixed
     $\varepsilon>0$, we have shown the existence of a solution
     $(n_{1,\varepsilon},n_{2,\varepsilon})$ to (\ref {S-reg}) such that for $i=1,2$
     \begin{equation}\label{eq:maxprinc}
     	0\le n_{i,\varepsilon}(t,x)\le M,
     \end{equation}
     for a.e. $(t,x)\in \Om_T$ where $M>0$ is a constant not depending on $\eps$.
     

     Taking $\psi_i=n_{i,\varepsilon}$,  $\varphi_i=w_{i,\varepsilon}$, $\Psi_i=U_{\varepsilon}$
     as test functions in (\ref{form:ueps}) and working exactly as in Lemma \ref{lem-classic:est},
     we obtain for $i=1,2$
     \begin{equation}
     	\label{est1-u-weak-bis}
     	\begin{split}\displaystyle
     		&\sup_{0 \le t \le T} \int_{\Omega}\abs{n_{i,\varepsilon}(t,x)}^2  \,dx +
     		\iint_{\Om_T}\abs{\nabla n_{i,\varepsilon}}^2  \,dx\,dt \le C,\\
     		&\iint_{\Om_T}\abs{\nabla n_{i,\varepsilon}}^2  \,dx\,dt \le C,\\
     		&\iint_{\Om_T}\abs{\nabla w_{i,\varepsilon}}^2  \,dx\,dt \le C,\\
     		&\sup_{0 \le t \le T} \int_{\Omega}\abs{U_{\varepsilon}(t,x)}^2  \,dx +
     		\iint_{\Om_T}\abs{\nabla U_{\varepsilon}}^2  \,dx\,dt \le C,
     	\end{split}
     \end{equation}
     for some constant $C>0$ independent of $\varepsilon$.
     
     Working exactly as the proof of $(iv)$ in Lemma \ref{lem-classic:est}, we get easily for $i=1,2$
     \begin{equation}\label{est2-u-weak-bis}
     	\norm{\partial_t n_{i,\varepsilon}}_{L^2(0,T;(H^1(\Omega))')}+ \norm{\partial_t U_{\varepsilon}}_{\L^1(0,T; \bV^\prime)}\le C,
     \end{equation}
     for some constant $C>0$ independent of $\varepsilon$. 
     
     Then, by (\ref {est1-u-weak-bis})-(\ref{est2-u-weak-bis}) and standard compactness
     results (see \cite{Aubin-Simon}) we can extract subsequences, which we do not relabel,
     such that, as $\varepsilon$ goes to $0$,
     \begin{equation}
     	\begin{array}{l}\displaystyle
     		n_{i,\eps} \to n_i
     		\mbox{ weakly-}\star \mbox{ in  } L^\infty(\Om_T),\\
     		n_{i,\eps} \to n_i \mbox{ weakly in }L^{2}(0,T;H^1(\Omega)), \\
     		n_{i,\eps} \to n_i
     		\mbox{ strongly} \mbox{ in  } L^2(\Om_T),\\
     		\partial_t  n_{i,\eps} \to \partial_t n_i \mbox{ weakly in }L^{2}(0,T;(H^1(\Omega))'),\\
     		w_{i,\eps} \to w_i \mbox{ weakly in }L^{2}(0,T;H^1(\Omega)), \\
     		U_\eps \to U\mbox{ weakly-}\star \mbox{ in  } \L^\infty\big(0,T; \bH \big),\\
     		U_\eps \to U\mbox{ weakly in }\L^2\big(0,T; \bV\big),\\
     		U_\eps \to U\mbox{ strongly} \mbox{ in  } \L^\infty\big(0,T; \bH \big),
     		\label{pa}
     	\end{array}
     \end{equation}
     for $i=1,2$. From the compact embedding $L^\infty(\Omega) \subset (H^1(\Omega))'$, we also have that
     $n_{i,\eps}$ is a Cauchy sequence in $C(0,T;(H^1(\Omega))')$ for $i=1,2$. Moreover,
     with the convergences \eqref{pa} and the weak-$\star$ convergence of $n_{i,\eps}$ to $n_i$
     in $L^\infty(\Om_T)$, we obtain
     $$
     n_{i,\eps} \to n_i \mbox{ strongly in $L^p(\Om_T)$ for $1\le p
     	<\infty$ and for $i=1,2$}.
     $$
     With the above convergences, we pass to the limit in \eqref{form:ueps} to obtain the weak formulation \eqref{wf-1} in the sense of Definition \ref{defSol}.\\
     \noindent In the following step, we define the operator $B$ such that $B(U):=(U\cdot\nabla  )U$ for $U\in \L^2\big(0,T; \bV\big)$. Note that we can write the equation of $U$ in \eqref{wf-1} in the following form
     \begin{equation}\label{Eq3.25}
     	\frac{d}{dt}\langle U,\Psi\rangle=\langle - \nu\Delta U+B(U)+Q(n_1,n_2)\nabla   \phi,\Psi\rangle,\;\;\forall \Psi\in\bV.
     \end{equation}
     Since the operator $-\Delta:\bV\to\bV^\prime$ is linear and continuous
     and $U\in\L^2(0,T;\bV)$, we deduce easily that $-\Delta U\in \L^2(0,T;\bV^\prime)$. Moreover,
     $Q(n_1,n_2)\nabla   \phi\in\L^2(0,T;\bV^\prime)$ and the operator $b(U,U,w)=\langle B(U),w\rangle$ is trilinear continuous on $\bV$. Furthermore, we exloit $\parallel B(U)\parallel_{\bV^\prime}\leq \parallel U \parallel_\bV$ to deduce $B(U)\in\L^1(0,T,\bV^\prime)$ and consequently we arrive to $\partial_t U\in\L^1(0,T,\bV^\prime)$.\\
     
     \noindent In the final step we are interested of the recuperation of the pressure $p$.
     For this we set
     $$
     \displaystyle I_1(t)=\int_{0}^{t}U(s)\,ds,\;I_2(t)=\int_{0}^{t}(U\cdot\nabla  )U(s)\,ds,\;I_3(t)=\int_{0}^{t}Q(n_1,n_2)(s)\nabla  \phi\,ds.
     $$
     It is clear that $I_1,\,I_2,\,I_2\,\in C(0,T;(\H^1(\Omega))^\prime)$.
     Integrating \eqref{Eq3.25} over $[0,T]$ yields
     $$
     \langle U(t)-U_0-\nu\Delta I_1(t)+I_2(t)+I_3(t),\Psi\rangle=\bzeros,\;\;\forall t\in[0,T], \;\;\forall \psi\in\bV.
     $$
     An application of the Rham Theorem (see \cite{Te01} for more details), there exists $P(t)\in\L^2_0(\Omega)$ such that
     $$
     \displaystyle U(t)-U_0-\nu\Delta I_1(t)+I_2(t)+I_3(t)+\nabla   P=\bzeros,
     $$ 
     for each $t\in[0,T]$, where $\displaystyle\L^2_0(\Omega)=\Big\{w\in\L^2(\Omega),\int_{\Omega}w\,dx=0\Big\}$.
     This implies that $\nabla   P\in C(0,T;\H^{-1}(\Omega))$ and thus $P\in C(0,T;\L^2_0(\Omega))$.
     Finally, a derivation with respect to $t$ in the sense of distributions, we obtain 
     $$\displaystyle\partial_t U-\nu\Delta U+(U\cdot\nabla  )U+Q(n_1,n_2) \nabla \phi+ \nabla p=\bzeros,
     $$
     where $p=\partial_t P\in W^{-1,\infty}(0,T;\L^2_0(\Omega))$. 
	\section{Multiscale derivation toward chemotaxis-chemicals in a fluid }\label{Sec3}
	This section is devoted to the derivation of chemotaxis-chemicals--fluid system  with predator prey terms from an kinetic--fluid model using  the micro-macro decomposition technique inspiring from \cite{ABKMZ20}. We start by presenting the kinetic--fluid model and its properties. Next, an equivalent appropriate system on the basis of the micro-macro decomposition method is obtained. Then, our system \eqref{CrossDiff} is derived. 
	
	We consider the case where the set for velocity is a sphere of radius $r>0$, $V=rS^{d-1}$. The kinetic--fluid model is given as follows  
	\begin{equation}\label{so2}
		\left\{
		\begin{array}{l}
			\displaystyle\varepsilon \partial_tf_1+ v \cdot \nabla_{x}f_1^\varepsilon=
			\frac{1}{\varepsilon}\mathcal{T}_1[f_{2}^\varepsilon](f_1^\varepsilon ) +G_1(f_1^\varepsilon,f_2^\varepsilon,w_1,w_2,v,U),    \\\\
			\displaystyle\varepsilon \partial_tf_2^\varepsilon+ v \cdot \nabla_{x}f_2=
			\frac{1}{\varepsilon}\mathcal{T}_2[f_{1}^\varepsilon](f_2^\varepsilon) +G_2(f_1^\varepsilon,f_2^\varepsilon,w_1,w_2,v,U),    \\\\
			\displaystyle U\cdot\nabla  w_1-\Delta w_1+\alpha_{1}\,w_1=\beta_1 \int_{V}f_2^\varepsilon dv,\\\\
			\displaystyle U\cdot\nabla  w_2-\Delta w_2+\alpha_{2}\,w_2=\beta_2\int_{V}f_1^\varepsilon dv,\\\\
			\displaystyle\partial_t U -\nu \Delta U+ k(U\cdot \nabla)U+\nabla p+Q\Bigg(\int_{V}f_1^\varepsilon dv,\int_{V}f_2^\varepsilon dv\Bigg) \nabla\phi = \bzeros,  \; \;  \vdiv U=0,\\ \\
			f_i^\varepsilon(t=0,x,v)=f^\varepsilon_{i,0}(x,v),\quad U(t=0,x)=U_{0}(x),
		\end{array}\right.
	\end{equation}
	where $f_1(t,x,v)$ and $f_2(t,x,v)$ are the distribution functions describing the statistical evolution of predator and prey species, where $t > 0$, $x\in \mathbb{R}^d,$ and $v\in V$ are time, position, and velocity, respectively, $\mathcal{T}_i$ is a stochastic operator representing a random modification of the direction of the predator and prey, and the operator $G_i$ describes the gain-loss balance of theses species. 
	To apply the micro-macro decomposition method by low order
	asymptotic expansions in term of the mean free path $\varepsilon$, the following assumptions are needed. The turning operator $\mathcal{T}_i$ is decomposed as follows:
	\begin{equation}
		\mathcal{T}_1[f_2^\varepsilon](f_1^\varepsilon)= \mathcal{L}_1(f_1^\varepsilon)+\varepsilon\,\mathcal{T}_1^2[f_1^\varepsilon](f_2^\varepsilon),
	\end{equation}
	\begin{equation}
		\mathcal{T}_2[f_1^\varepsilon](f_2^\varepsilon)= \mathcal{L}_2(f_2^\varepsilon)+\varepsilon\,\mathcal{T}_2^2[f_2^\varepsilon](f_1^\varepsilon),
	\end{equation}
	where $\mathcal{L}_1$, ($\mathcal{L}_2$) represents the dominant part of the turning kernel and is assumed to be independent of $f_2^\varepsilon$, ($f_2^\varepsilon$) respectively. Herein, we omit the dependence on $\varepsilon$ in the functions $f_1^\varepsilon$ and $f_2^\varepsilon$.

	The operators $\mathcal{T}_{i}$ for  $i,j=1,2$ are given by
	\begin{equation}
		\displaystyle\mathcal{T}^j_{i}(f_i)= \int_{V} \big(T^{j}_{i}(v^*,v)f_i(t, x, v^{*}) -
		{T^j_{i}} (v,v^*)f_i(t, x, v) \big)dv^{*}, \label{3.42}
	\end{equation}
	where $T_i^j$ is the probability kernel for the new velocity $v\in V$ given that the previous velocity was $v^*$. We assume that $T_i^1=\frac{\sigma_i}{\abs{V}}$. Then,
	\begin{equation}\label{Ti}
		\mathcal{T}_i^1(g)=-\sigma_i\,g.
	\end{equation} 
	Remark that the operators $\mathcal{L}_i(g)$ and $\mathcal{T}^j_i$ satisfy
	\begin{equation}\label{H02}
		\displaystyle \int_V \mathcal{L}_i(g)\,dv=\int_V \mathcal{T}^j_i(g)\,dv=0, \;\;i,j=1,2.
	\end{equation}
	Moreover, there exists a bounded velocity distribution $M_i(v)>0$ for $i=1,2$ independent of $t$ and $x$ such that
	\begin{equation}\label{tx2}
		T_i^1 (v,v^{*} ) M_i(v^{*}) = T_i^1 (v^{*},v ) M_i(v),
	\end{equation}
	holds. We consider the following choice
	\begin{equation}\label{M}
		M_i(v) = \frac{1}{\abs{V}}.
	\end{equation}
	Note that the flow produced by these equilibrium
	distributions vanishes and $M_i$ are normalized, i.e.
	\begin{equation}
		\int_V v \, M_i(v)dv  =0, \quad \int_V
		M_i(v)dv =1, \quad  i=1,2. \label{equilibre2}\end{equation}

	The other probability kernel $T_i^2$ is given by
	\begin{equation}\label{2.8}
		T_i^2[f_2](v,v^*)=\frac{\sigma_i\, D_i\,M_i}{f_i}\big(1+d_i(f_i)\big)\,v\cdot\nabla\Big (\frac{f_i}{M_i}\Big).
	\end{equation}
	
	Second, the interaction operators $G_i$ satisfy the following properties	
	\begin{equation} \label{D2}\displaystyle G_i(f_1,f_2,w_1,w_1,v,U)= G_{i}^1(f_1,f_1,w_1,w_1,v,U)+ \varepsilon\,
		G_{i}^2(f_1,f_2), 
	\end{equation} where
	
	\begin{equation}\label{G11}\displaystyle G_{i}^1(f_1,f_2,w_1,w_2,v,U) =\frac{d\,\sigma_i}{r^2\,\abs{V}}\Big(f_i\,U+\chi_i\Big(\int_{V}f_idv\Big)\nabla w_i\Big).
	\end{equation}
	
	Note that \begin{equation}\label{CO2}\displaystyle\int_V G_{i}^1(f_1,f_2,w_1,w_2,v,U)dv =0,\;\;i = 1,2.
	\end{equation}
	We define the interactions operators $G_{2}^1$ and $G_{2}^2$ by
	\begin{equation}\label{G2}
		\displaystyle G_{1}^2(f_1,\,f_2)= \frac{1}{|V|}f_1(a_1-b_1f_1-c_1f_2),
		\qquad \,\,  G_2^2(f_1,\,f_2)= \frac{1}{|V|}f_2(a_2-b_2f_1-c_2f_2).
	\end{equation}
	
	\noindent Using the same arguments as in \cite{ABKMZ20}, we find that the operator $\mathcal{L}_i$ has the following properties.
	\begin{Lemma}
		\label{LE1}  The following properties of the operator $\mathcal{L}_i$ for $i=1,2$ holds true:
		\begin{itemize}\label{L1}
			\item[i)] The operator $\mathcal{L}_i$ is self-adjoint in the space
			$\displaystyle{{\L^{2}\left(V ,{d\xi \over M_i}\right)}}$. \item[ii)]
			For $f\in \L^2$, the equation $\mathcal{L}_i(g) =f$ has a unique
			solution $\displaystyle{g \in \L^{2}\left(V, {d\xi \over
					M_i}\right)}$, satisfying
			$$
			\int_{V} g(\xi)\, d\xi = 0 \quad \Longleftrightarrow \quad   \int_{V} f(\xi)\, d\xi =0. 
			$$
			\item[iii)]  The equation $\mathcal{L}_i(g) =\xi \,  M_i(\xi)$, has a
			unique solution denoted by $\theta_i(\xi)$ for $ i=1,2$. 
			\item[iv)]  The kernel
			of $\mathcal{L}_i$ is $N(\mathcal{L}_i) = vect(M_i(\xi))$ for $ i=1,2$.
		\end{itemize}
	\end{Lemma}
	
	\noindent We denote the integral with respect to the variable $v$ will be denoted by $\langle . \rangle$. The main idea of the micro-macro method is to decompose the distribution function $f_i$ for $i=1,2$ as follows
	$$f_i(t,x,v)=M_i(v) n_i(t,x) + \varepsilon  g_i(t,x,v), $$
	where 
	$$n_i(t,x)= \langle f_i(t,x,v)\rangle:=\int_Vf_i(t,x,v)\,dv.$$ 
	This implies that $\langle g_i \rangle=  0$ for $ i=1,2$. Inserting $f_i$ in kinetic--fluid model (\ref{so2}) and using the above assumptions and properties of the interaction and the turning operators, one obtains
	\begin{equation}\label{n2} 
		\left\{
		\begin{array}{l}
			\displaystyle\partial_t (M_1 (v)n_1)  + \varepsilon \partial_t g_1 +
			\frac{1}{\varepsilon}  v M_1(v) \cdot \nabla n_1 + v \cdot
			\nabla g_1   = \frac{1}{\varepsilon}\mathcal{L}_1(g_1)\\ 
			{}\\
			\displaystyle\hskip0.8cm+\mathcal{T}_1[f_2](f_1)+\frac{1}{\varepsilon}G^{1}_{1}(f_1,f_2,w_1,w_2,v,U)+G^{2}_{1}(f_1,f_2),\\
			\\{}
			\displaystyle\partial_t (M_2 (v)n_2)  + \varepsilon \partial_t g_2 +
			\frac{1}{\varepsilon}  v M_2(v) \cdot \nabla n_2 + v \cdot
			\nabla g_2   = \frac{1}{\varepsilon}\mathcal{L}_2(g_2)\\ 
			{}\\
			\displaystyle\hskip0.8cm+\mathcal{T}_2[f_1](f_2)+\frac{1}{\varepsilon}G^{2}_{1}(f_1,f_2,w_1,w_2,v,U)+G^{2}_{2}(f_1,f_2),\\{}\\
			\displaystyle U\cdot\nabla  w_1-\Delta w_1+\alpha_{1}\,w_1=\beta_1 n_2,\\\\
			\displaystyle U\cdot\nabla  w_2-\Delta w_2+\alpha_{2}\,w_2=\beta_2n_1,\\\\
			\displaystyle\partial_t U -\nu \Delta U+ k(U\cdot \nabla)U+ \nabla p+Q\big(n_1,n_2\big) \nabla\phi = \bzeros,  \; \;  \vdiv U=0.
		\end{array}
		\right.
	\end{equation}
	
	\noindent In order to separate the macroscopic density $n_i(t,x)$ and microscopic quantity $g_i(t,x,v)$ for $i=1,2$ one has to use the projection technique. For that, let consider $P_{M_i}$ the orthogonal projection onto $N(\mathcal{T}_i)$, for $i=1,2$. It follows
	$$P_{M_i}(h)= \langle h\rangle M_i, \quad  \mbox{for any}\quad  h\in
	\displaystyle{{\L^{2}\left(V ,{dv \over M_i}\right)}}, \qquad i=1,2.$$ 
	
	\noindent Now, inserting the operators $I -P_{M_i}$ into Eq. \eqref{n2}, and using known properties for the projection $P_{M_i},$ $ = 1,2$ yields the following micro-macro formulation  
	\begin{equation}\label{mM12}
		\left\{
		\begin{array}{l l}
			\displaystyle\partial_t g_1 +
			\frac{1}{\varepsilon^2} v M_1(v) \cdot \nabla n_1+ \frac{1}{\varepsilon}(I-P_{M_1})(v
			\cdot
			\nabla g_1) =\frac{1}{\varepsilon^2}\mathcal{L}_1(g_1)\\
			\displaystyle \hspace{0.5cm}+\frac{1}{\varepsilon}\mathcal{T}_1[f_2](f_1)+\frac{1}{\varepsilon^2}G_{1}^{1}(f_1,      f_2,w_1,w_2,v,U)+ \frac{1}{\varepsilon}(I-P_{M_1}) G_{1}^2(f_1,f_2), \\
			{}\\
			\displaystyle\partial_t n_1+ \,\langle v   \cdot
			\nabla g_1 \rangle =  \langle G^{2}_{1}(f_1,f_2)\rangle, \\ {} \\
			\displaystyle\partial_t g_2 +
			\frac{1}{\varepsilon^2} v M_2(v) \cdot \nabla n_2+ \frac{1}{\varepsilon}(I-P_{M_2})(v
			\cdot\nabla g_2) =\frac{1}{\varepsilon^2}\mathcal{L}_2(g_2)\\
			\displaystyle \hspace{0.5cm}+\frac{1}{\varepsilon}\mathcal{T}_2[f_1](f_2)+\frac{1}{\varepsilon^2}G_{2}^{1}(f_1,      f_2,w_1,w_2,v,U)+ \frac{1}{\varepsilon}(I-P_{M_2}) G_{2}^2(f_1,f_2), \\
			{}\\
			\displaystyle\partial_t n_2+ \,\langle v   \cdot
			\nabla g_2 \rangle =  \langle G^{2}_{2}(f_1,f_2)\rangle, \\ {} \\
			\displaystyle U\cdot\nabla  w_1-\Delta w_1+\alpha_{1}\,w_1=\beta_1 n_2,\\\\
			\displaystyle U\cdot\nabla  w_2-\Delta w_2+\alpha_{2}\,w_2=\beta_2n_1,\\\\
			\displaystyle\partial_t U -\nu \Delta U+ k(U\cdot \nabla)U+ \nabla p+Q\big(n_1,n_2\big) \nabla\phi = \bzeros,  \; \;  \vdiv U=0.
		\end{array} \right.
	\end{equation}
	
	The following proposition states that the micro-macro formulation \eqref{mM12} is equivalent to  kinetic-fluid model (\ref{so2})
	
	\begin{Proposition}
		
		\noindent i) Let $(f_1,f_2,w_1,w_2,U,p)$ be a solution of nonlocal kinetic-fluid model (\ref{so2}). Then \\$(n_1,n_2,g_1,g_2,w_1,w_2,U,p)$ (where
		$n_i=\langle f_i \rangle $ and $g_i=
		{1\over \varepsilon}(f_i-M_i n_i)$) is a solution to coupled system
		(\ref{mM12}) associated with the following initial data for $i=1,2$
		\begin{equation} \label{er42} n_i(t=0)=n_{i,0} =\langle f_{i,0} \rangle, \quad
			g_i(t=0)=g_{i,0}={1\over \varepsilon}(f_{i,0}-M_i n_{i,0}),\quad \hbox{and}\; U(t=0)=U_{0}, 
		\end{equation}	
		\noindent ii) Conversely, if $(n_1,n_2,g_1,g_2,w_1,w_2,U,p)$ satisfies system
		(\ref{mM12}) associated with the following initial data \\ $(n_{1,0},n_{2,0}, g_{1,0},g_{2,0},U_{0})$ such
		that $\langle g_{i,0} \rangle=0$ for $i=1,2$. Then $(f_1,f_2,w_1,w_2,U,p)$ (where $f_i=M_i n_i+\varepsilon g_i$)  is a solution to nonlocal kinetic-fluid model (\ref{so2}) with initial data
		$f_{i,0}=M_i n_{i,0}+\varepsilon g_{i,0}$ and one has $ n_i=\langle f_i \rangle$ and $\langle g_i\rangle=0$, for $i=1,2$.
	\end{Proposition} 
	
	\noindent Next, in order to develop asymptotic analysis of
	system (\ref{mM12}), $\mathcal{T}_i$ and $G_{i}^j$ assumed to satisfy the following asymptotic behavior $\varepsilon \to 0$
	
	\begin{equation}
		\mathcal{T}_1[M_2n_2+\varepsilon g_2]=\mathcal{T}_1[M_2n_2]+O(\varepsilon),\qquad 	\mathcal{T}_2[M_1n_1+\varepsilon g_1]=\mathcal{T}_2[M_1n_1]+O(\varepsilon),
	\end{equation}
	\begin{equation} \label{G22} G_{i}^{1}\big(M_1n_1 +\varepsilon g_1, M_2n_2 +\varepsilon g_2,w_1,w_2,v,U\big)= G_{i}^{1}\big(M_1n_1, M_2n_2,w_1,w_2,v,U \big)+ O(\varepsilon), 
	\end{equation} 
	and 
	\begin{equation} \label{G23} G_{i}^{2}\big(M_1n_1 +\varepsilon g_1, M_2n_2 +\varepsilon g_2\big)= G_{i}^{2}\big(M_1n_1, M_2n_2 \big)+ O(\varepsilon),
	\end{equation} for $ i=1,2$.
	\noindent Using assumptions \eqref{Ti}, \eqref{G22} and \eqref{G23}, the following equations for $g_i$ can
	be obtained from \eqref{mM12}
	\begin{eqnarray}\label{x12} g_i =  \mathcal{L}_i^{-1}\Big(v M_i \cdot \nabla n_i-\mathcal{L}_i^{-1}\Big(G_i^{1}(M_1 n_1, M_2 n_2,w_1,w_2,v,U)\Big)+O(\varepsilon),\;\;i=1,2.
	\end{eqnarray}
	\noindent Finally, inserting \eqref{x12} into the second and the fourth equations in \eqref{mM12}, yields macro--fluid model
	\begin{equation}\label{mM222}
		\left\{
		\begin{array}{l}
			\partial_t n_i +  \vdiv \, \Big( \beta_i(n_i)+\Gamma_i(n_1,n_2,w_1,w_2,U) -  D_i \cdot \nabla n_i\Big)= H_i(n_1,n_2) +O(\varepsilon),\\
			{}\\
			\displaystyle U\cdot\nabla  w_1-\Delta w_1+\alpha_{1}\,w_1=\beta_1 n_2,\\\\
			\displaystyle U\cdot\nabla  w_2-\Delta w_2+\alpha_{2}\,w_2=\beta_2n_1,\\\\
			\displaystyle  \partial_t U -\nu \Delta U+ k(U\cdot \nabla)U+ \nabla p+Q(n_1,n_2) \nabla\phi = \bzeros,  \; \;  \vdiv U=0,
		\end{array}
		\right.  
	\end{equation}
	\noindent where $D_i$, $\beta_i$, $\Gamma_i$ and $H_i$ are given, respectively, as follows
	\begin{equation}\label{di2}
		\qquad D_i =- \Big\langle v \otimes \theta_i(v) \Big\rangle =\frac{r^2}{d\,\sigma_i},
	\end{equation}
	with $\theta_i$ is given by
	$$\theta_i= - \frac{1}{\sigma_i}v  M_i(v)$$
	\begin{equation} \label{beta} 
		\beta_1(n_1)=  -\Big\langle  {\theta_1n_1\over
			M_1}\mathcal{T}_{1}^{2}[n_1](M_2) \Big\rangle=\frac{D_1}{n_1}\big(1+d_1(n_1)\big)\cdot\nabla n_1,
	\end{equation}
	
	\begin{equation} \label{beta2} 
		\beta_2(n_2)=  -\Big\langle  {\theta_2n_2\over
			M_2}\mathcal{T}_{2}^{2}[n_2](M_1) \Big\rangle=\frac{D_2}{n_2}\big(1+d_2(n_2)\big)\cdot\nabla n_2,
	\end{equation}
	\begin{equation} \label{w32} 
		\Gamma_i(n_1,n_2,w_1,w_2,U)=  -\Big\langle  {\theta_i\over
			M_i}G_{i}^{1}(M_1 n_1, M_2 n_2,w_1,w_2,v,U) \Big\rangle=n_iU+\chi_i(n_i)\nabla w_i,
	\end{equation}
	and
	\begin{equation} \label{H2}H_i(n_1,n_2)= \Big\langle G_{i}^{2}(M_1 n_1, M_2 n_2) \Big\rangle=F_i(n_1,n_2),\;  \hbox{for}\; i=1,2. 
	\end{equation}
	Finally, collecting the previous results with $\vdiv \,U=0$ and (\ref{mM222}), yields the macro-scale system  \eqref{CrossDiff} of the order $O(\varepsilon)$
	\begin{equation}\label{CrossDiff2}
		\left\{
		\begin{array}{l}
			\displaystyle \partial_t n_{1}+U\cdot\nabla n_1- \vdiv\Big(d_{1}(n_1) \nabla  n_1\Big) +\vdiv \Bigl(\chi_1(n_1)\nabla 
			w_1 \Bigl)=F_1(n_1,n_2)+O(\varepsilon),\\\\
			\displaystyle \partial_t n_{2}+U\cdot\nabla  n_2- \vdiv\Big(d_{2}(n_2) \nabla  n_2\Big)+ \vdiv \Bigl(\chi_2(n_2)\nabla 
			w_2 \Bigl)=F_2(n_1,n_2)+O(\varepsilon),
			\\ \\
			U\cdot\nabla  w_1-\Delta w_1+\alpha_{1}\,w_1=\beta_1 n_2,\\\\
			U\cdot\nabla  w_2-\Delta w_2+\alpha_{2}\,w_2=\beta_2 n_1,\\\\
			\displaystyle  \partial_t U -\nu \Delta U+ k(U\cdot \nabla)U+ \nabla p+Q(n_1,n_2) \nabla\phi = \bzeros,  \; \;  \vdiv U=0.
		\end{array}
		\right.\end{equation}
	
	\section{Computational analysis in two dimensions}\label{Sec4}
	
	We investigate computational analysis of nonlinear cross-diffusion--fluid with chemicals system \eqref{CrossDiff} in two dimensional space for two interacting populations; for instance, phytoplankton and zooplankton. First, we numerically demonstrate the cross-diffusion with chemicals in the absence of the fluid ($U=\bzeros$) by using the finite-volume method. Second, we consider the full system ($U\neq\bzeros$).
	We show the effect of external forces (obstacle inside the domain and the force of gravity) on the dynamics of fluid flow
	and simultaneously on the behavior of interacting populations by using the finite-element method.

	\subsection{Cross-diffusion with chemicals in the absence of fluid}
	We investigate two dimensional space computational analysis of nonlinear cross-diffusion with chemicals system \eqref{CrossDiff} using finite volume method. For that, we consider a family $\mathfrak{T}_h$ of admissible meshes of the domain $\Omega$ consisting of disjoint open and convex polygons called control volumes, see \cite{EGH00}. In the rest of this subsection, we shall use the following notation: the parameter $h$ is the maximum diameter of the control volumes in $\mathfrak{T}_h$. $K$ is a generic volume in $\mathfrak{T}$, $|K|$is the $2$-dimensional Lebesgue measure of $K$ and $N(K)$ is the set of the neighbors of $K$. Moreover, for all $L \in N(K)$, we denote by $\sigma_{K,L}$ the interface between $K$ and $L$ where $L$ is a generic neighbor of $K$. $\eta_{K,L}$ is the unit normal vector to $\sigma_{K,L}$ outward to $K$. For an interface $\sigma_{K,L}$, $|\sigma_{K,L}|$ will denote its $1$-dimensional measure. {$d_{K,L}$ denotes} the distance between $x_K$ and $x_L$, where the points $x_K$ and $x_L$ are respectively the center of $K$ and $L$.
	On the other hand, we assume that a discrete function on the mesh $\mathfrak{T}_h$ is a set $(g_K)_K\in\mathfrak{T}$ and we identify it with the piece-wise constant function $g_h$ on $\Omega$ such that $g_h\mid_K = g_K$. Furthermore, we consider an admissible discretization of $(0,T)\times\Omega$ consisting of an admissible mesh $\mathfrak{T}_h$ of $\Omega$ and of a time step size $\Delta t_h > 0$ (both $\Delta t_h$ and the size $\max_{K\in t_h}diam(K)$ tend to zero as $h \to 0$). Next, we define the discrete gradient $\nabla_hg_h$ as the constant per diamond $T_{K,L}$ function by
	\begin{equation*}\label {14}
		\Big(\nabla_hg_h\Big)\rvert_{\mathfrak{T}_{K,L}}=\nabla_{K,L}g_h:=\frac{g_L-g_K}{d_{K,L}}\eta_{K,L}.
	\end{equation*}
	Finally, we define the average of source terms $F_{i,K}^{k+1}$ by $F_{i,K}^{k+1}=F_i(n_1(t^k,x),n_2(t^k,x)),$ for $i=1,2$. And we make the following choice to approximate the diffuse terms 
	$$n_{i,K,L}^{k+1}=\min\{n_{i,K}^{{k+1}^+},n_{i,L}^{{k+1}^+}\},$$
	where $n_{i,J}^{{k+1}^+}=\max(0,n_{i,J}^{k+1})$ for $i=1,2$ and $J=K,L$.
	The computation starts from the initial cell averages $\displaystyle n_{i,0}^K=\frac{1}{|K|}\int_{K}n_{i,0}(x)\,dx$ for $i=1,2$.
	In order to advance the numerical solution from $t^k$ to $t^{k+1} = t^k + \Delta t$, we use the following implicit finite
	volume scheme: determine $n^{k+1}_{i,K}$ for $K\in \mathfrak{T}$, $i = 1, 2$ such that
	\begingroup
	\begin{equation}\label{2DScheme}
		\left\{
		\begin{array}{ll}
			|K|\frac{n^{k+1}_{1,K}-n^{k}_{1,K}}{\Delta t}-d_{1}\sum\limits_{L \in N(K)}	\frac{|\sigma_{K,L}|}{d_{K,L}}(n_{1,L}^{k+1}-n_{1,K}^{k+1})+\sum\limits_{L \in N(K)}\frac{|\sigma_{K,L}|} {d_{K,L}}\Big[\chi_1(n_{1,K,L}^{k+1})(w_{1,L}^{k+1}-w_{1,K}^{k+1})\Big] =|K| F_{1,K}^{k+1},\\
			|K|\frac{n^{k+1}_{2,K}-n^{k}_{2,K}}{\Delta t}-d_{2}\sum\limits_{L \in N(K)}	\frac{|\sigma_{K,L}|}{d_{K,L}}(n_{2,L}^{k+1}-n_{2,K}^{k+1})+\sum\limits_{L \in N(K)}\frac{|\sigma_{K,L}|} {d_{K,L}}\Big[\chi_2(n_{2,K,L}^{k+1})(w_{2,L}^{k+1}-w_{2,K}^{k+1})\Big]=|K| F_{2,K}^{k+1},\\
			\sum\limits_{L \in N(K)}	\frac{|\sigma_{K,L}|}{d_{K,L}}(w_{1,L}^{k+1}-w_{1,K}^{k+1})+\alpha_1|K|w_{1,K}^{k+1}=\beta_1|K|n_{2,K}^k,\\
			\sum\limits_{L \in N(K)}	\frac{|\sigma_{K,L}|}{d_{K,L}}(w_{2,L}^{k+1}-w_{2,K}^{k+1})+\alpha_2|K|w_{2,K}^{k+1}=\beta_2|K|n_{1,K}^k
		\end{array}\right.
	\end{equation}\endgroup
	for all $K\in \mathfrak{T}_h,\; k\in N_h$.  We consider implicitly the homogeneous Neumann boundary condition. To solve the corresponding nonlinear system arising from the implicit finite volume scheme \eqref{2DScheme}, we have used the Newton method. Note that the linear systems involved in Newton's method are solved by the GMRES method.
	
	For the numerical simulations, we consider uniform mesh giving by a Cartesian grid $N_x = N_y = 256$ and we take the following parameters  $$a_1 =10,\;a_2 = 0.1,\;b_1 = b_2 = 2,\;c_1 = 0.4,\;c_2 = 0.01.$$ 
	
	The corresponding diffusion coefficients are given by $d_i  = \alpha_i = 1$, for $i = 1, 2$. The chemotactic sensitivity parameters are chosen by
	$$\beta_1=20,\;\beta_2=100.$$
	
	\subsubsection{ Example 1: species the interacting via chemical substance} 
	
	For this numerical test, the chemotactic coefficients are $\chi_1 =2>0,$ and $\chi_2=-0.8<0$. This matches well cross-diffusion phenomena where the predator directs its movement towards the prey, while the movement of the prey is against the presence of the predator.  For the initial condition, the prey and predator  are concentrated in small pockets at a four spatial point (see Figure \ref{fig1}).
	
	In Figure \ref{fig1}, we display the numerical solution for each species at four different simulated times. Initially, at time $t = 0.05$, we can observe the effect of the chemotaxis for the predator feeling their prey, and the prey feeling the presence of the predator. At time $t = 0.1$. We notice the rapid movement of the predator towards the regions occupied by the prey. The prey moves to the regions where the predator is not located. At time $t = 0.5$, it is clearly seen that the predator occupy almost the entire area, while the prey moves toward (running away) the area where the predator is not located.
	
	\begin{figure}[pos=!ht]
		\centering
		{\includegraphics[height=2in ,width=6in]{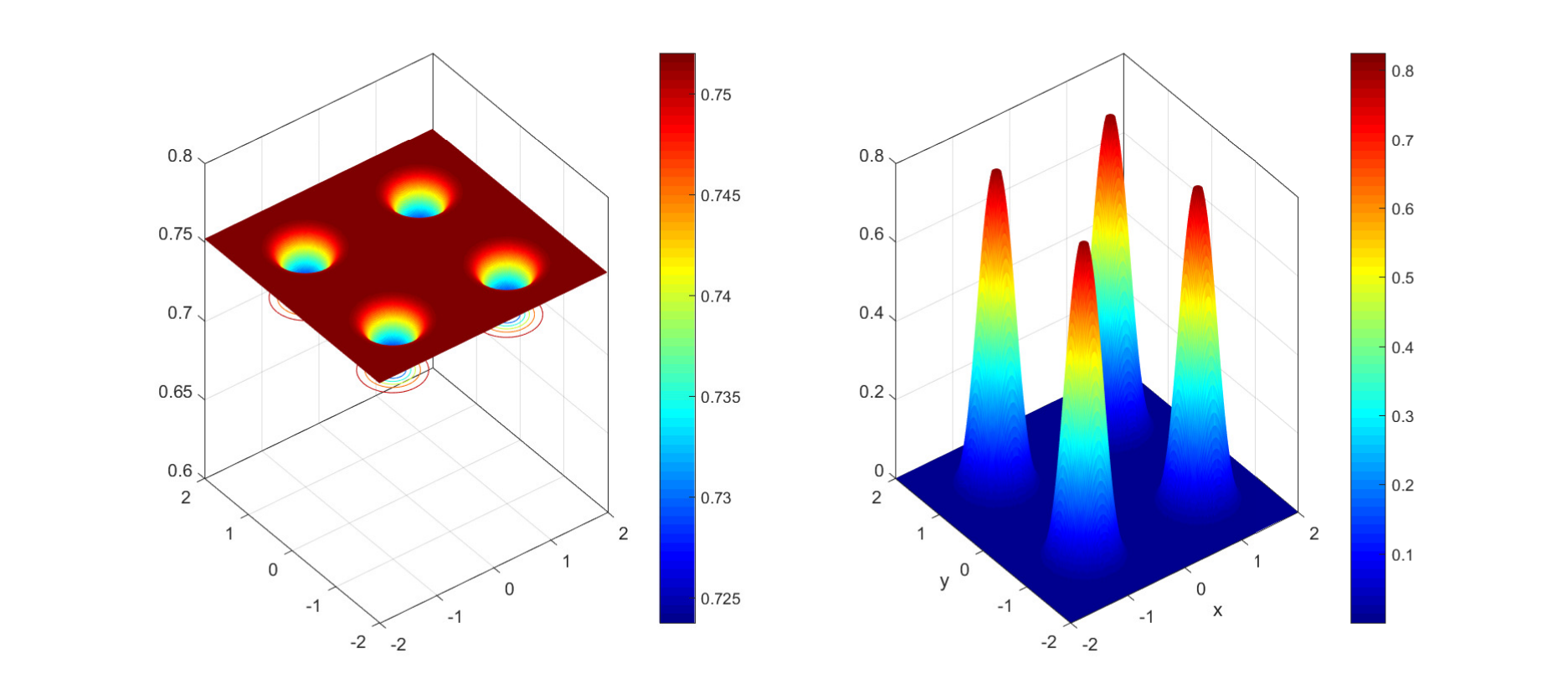}}
		{\includegraphics[height=2in ,width=6in]{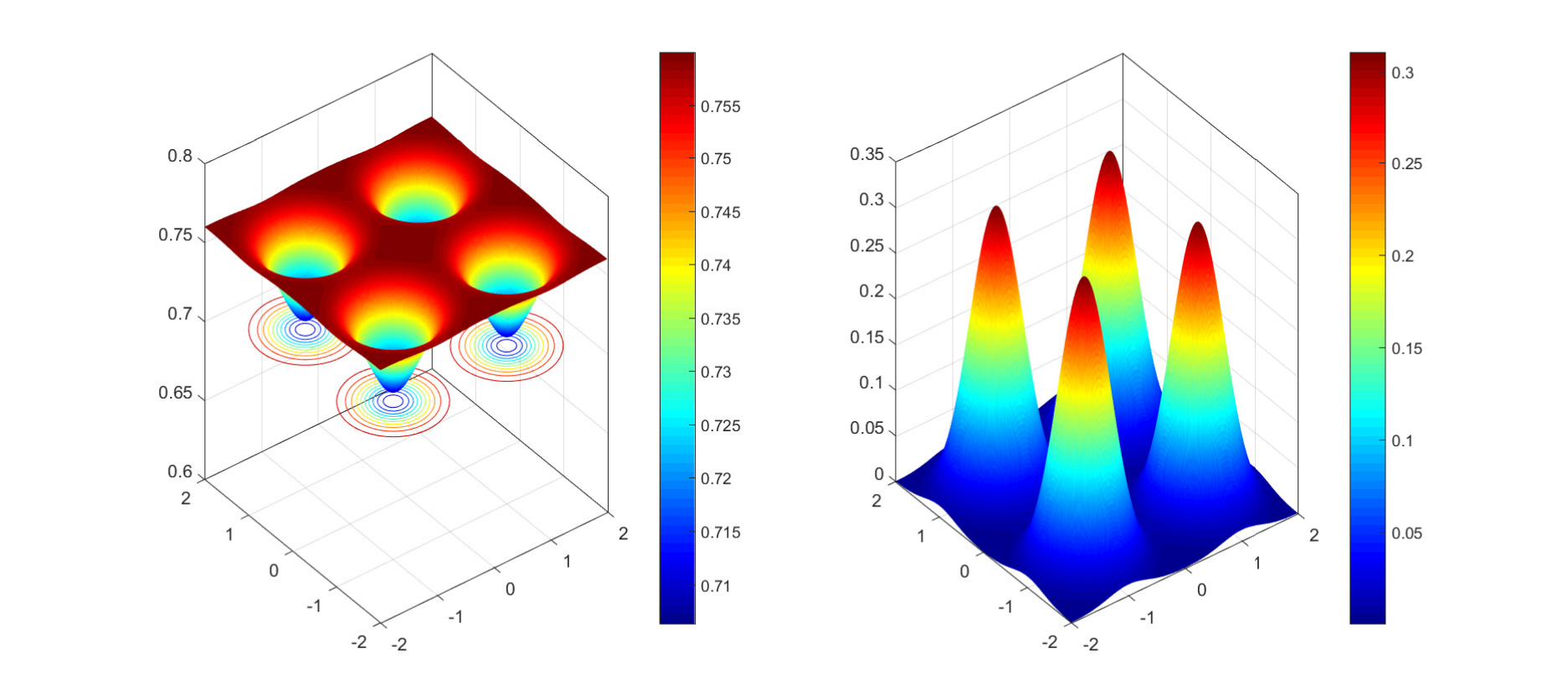}}
		{\includegraphics[height=2in ,width=6in]{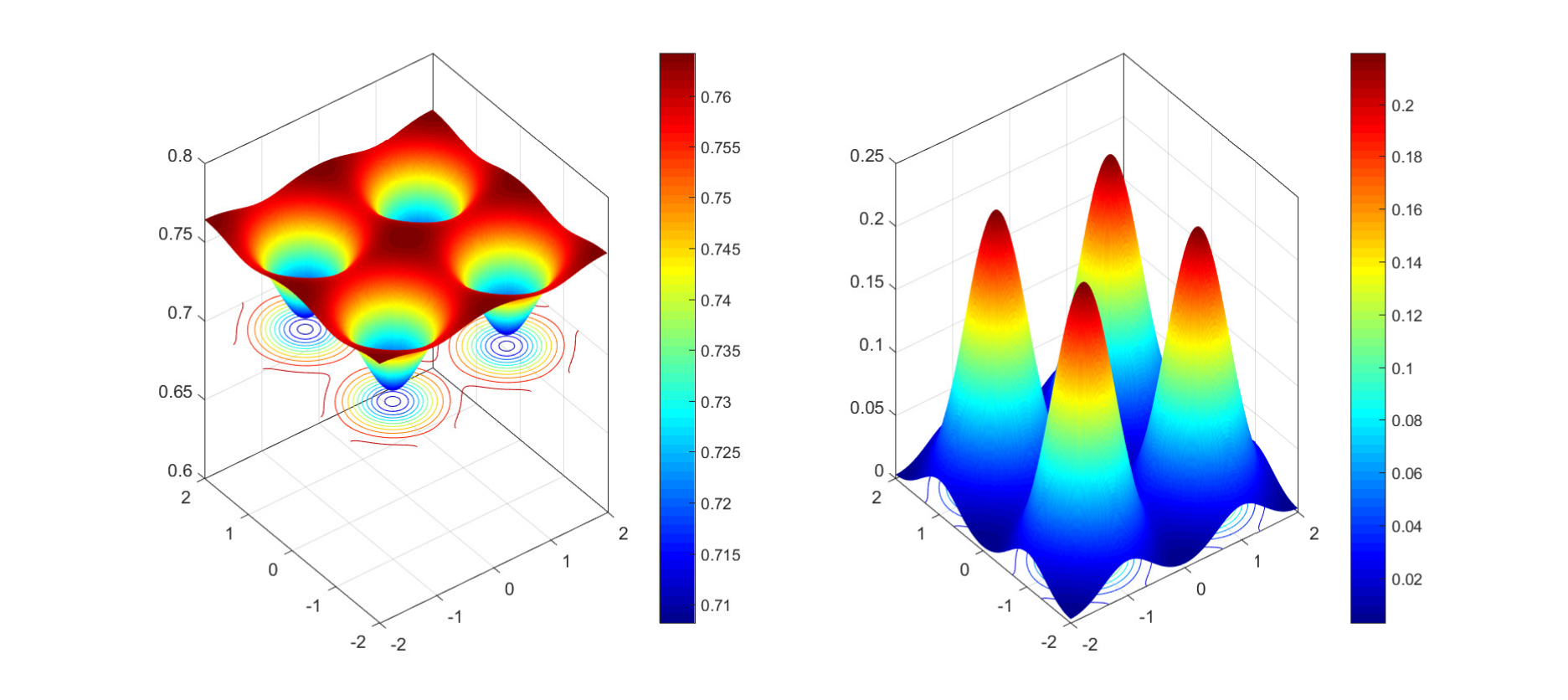}}
		{\includegraphics[height=2in ,width=6in]{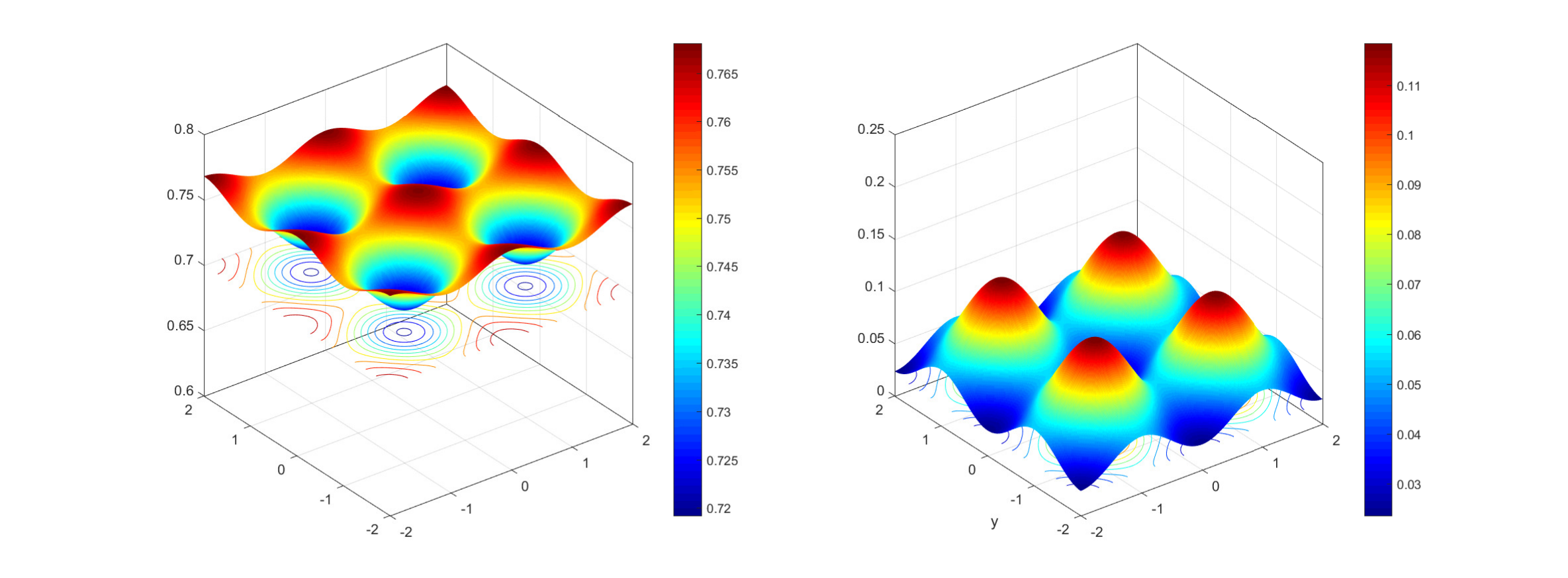}}
		\caption{Example 1: interaction of the predator and the prey at different times $t = 0.01,\, 0.05,\, 0.75,\, 0.15$}
		\label{fig1}
	\end{figure}
	
	\subsubsection{Example 2: prey do not interact via chemical substances }
	
	In this Example, we consider $\chi_1 = 2$ and $ \chi_2= 0$. This means that we do not consider chemotactic movement of the prey. The predator and prey are concentrated in small pockets at a one spatial point (see Figure \ref{fig2}). We show in Figure \ref{fig2} the numerical solution for each species at four different simulations time. We notice the rapid movement of the predator spreads out to the areas where the prey is located, while the prey presents isotropic and homogeneous diffusion (due to the choice of the tactic coefficient).
	\begin{figure}[pos=!ht]
		\centering
		{\includegraphics[height=2in ,width=6in]{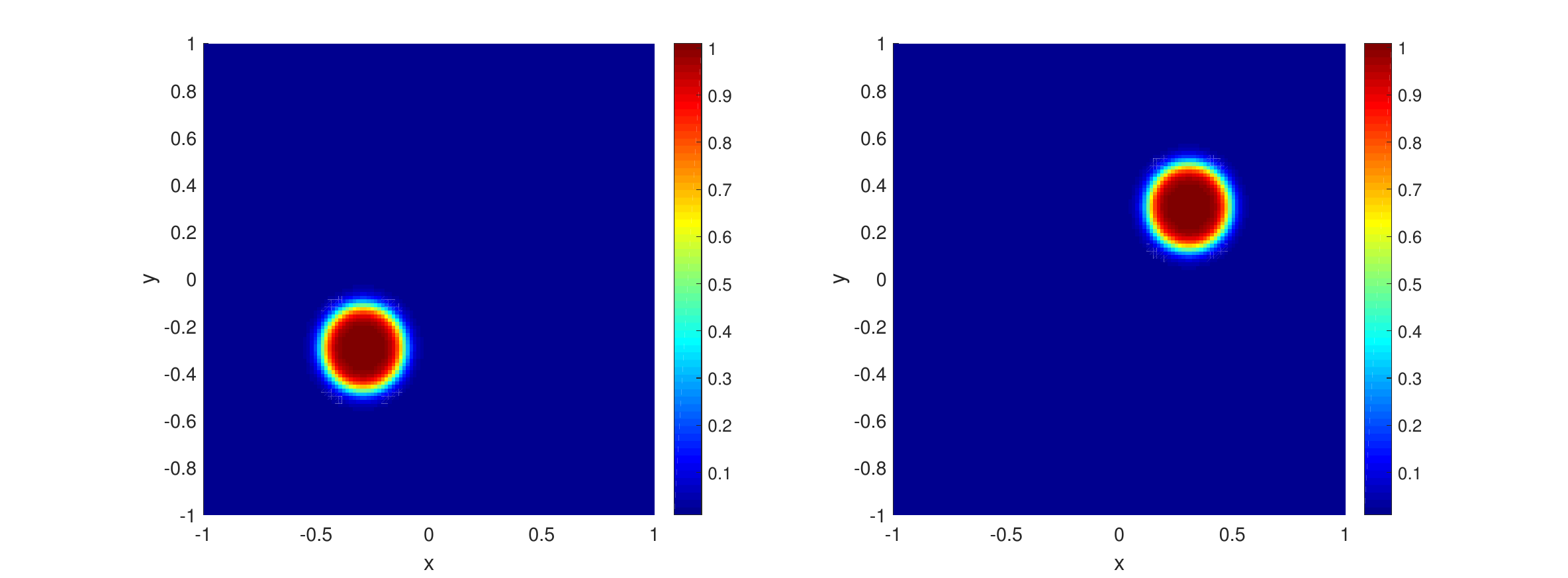}}
		{\includegraphics[height=2in ,width=6in]{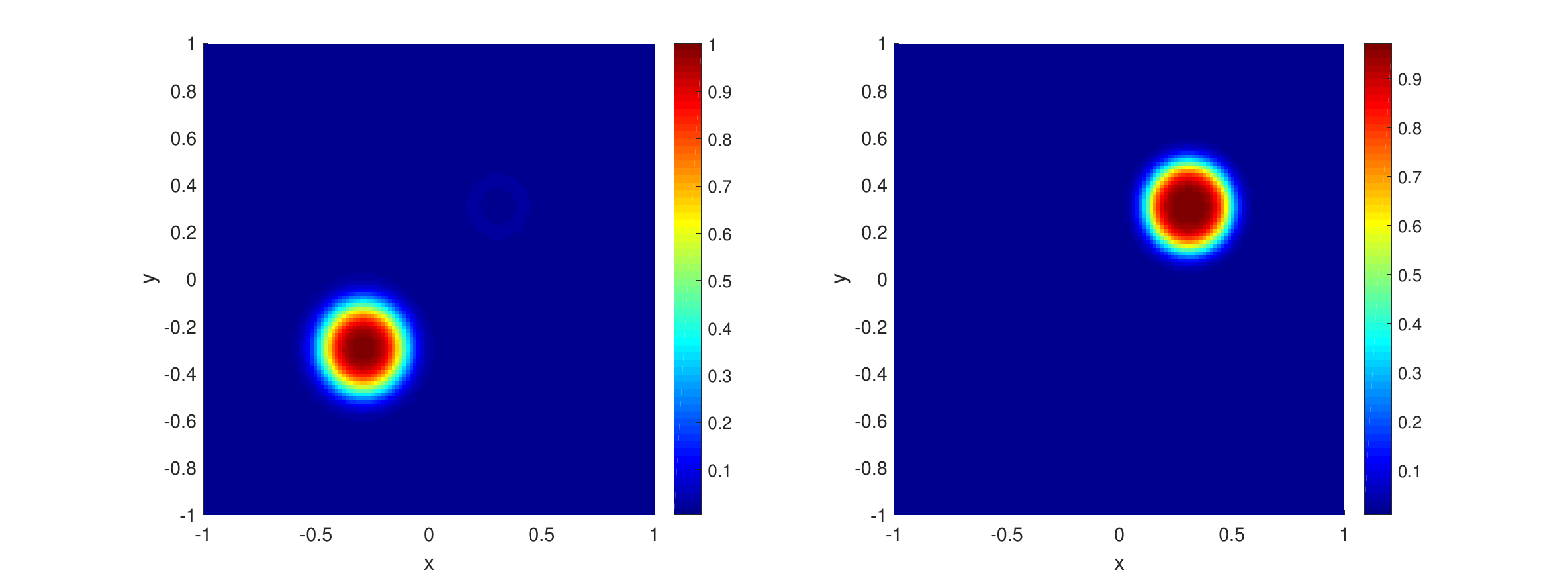}}
		{\includegraphics[height=2in ,width=6in]{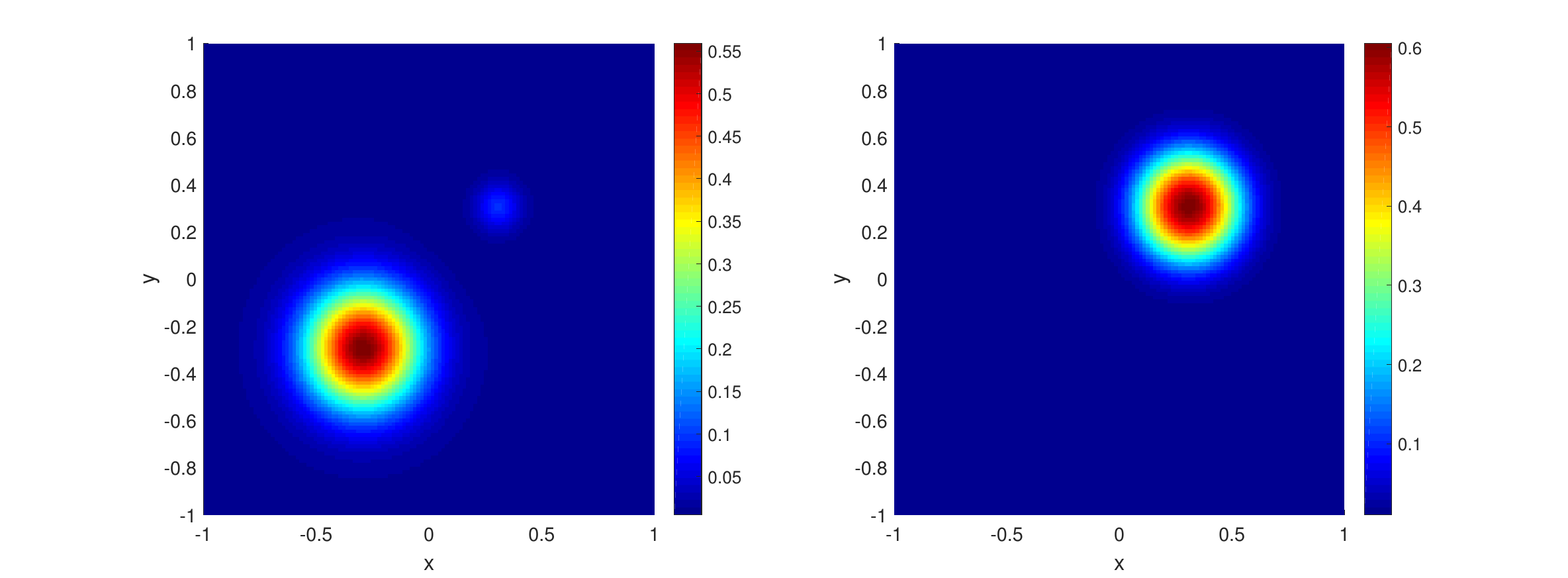}}
		{\includegraphics[height=2in ,width=6in]{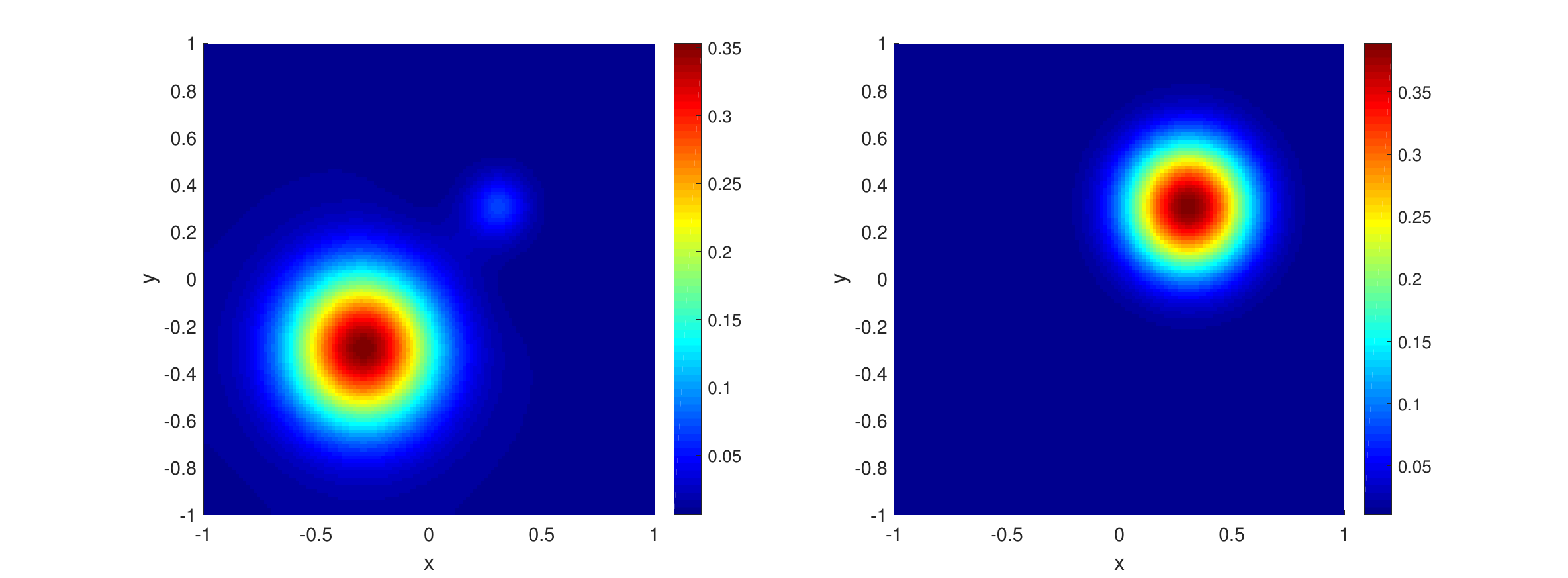}}
		\caption{Example 2: interaction of the predator and the prey at different times $t = 0, 0.05,\, 0.1,\, 0.5$}
		\label{fig2}
	\end{figure}
	
	\subsubsection{Example 3: spatial patterns formation}  We assume that the densities of species are a random perturbation around the stationary state $(n_1^*,n_2^*)$. Consequently, the initial data are given by
	$$n_1(0,x)=n_1^*+n_1(x)_\delta,\quad n_2(0,x)=n_2^*+n_2(x)_\delta,\,\qquad x\in\Omega,$$
	where $J(x)_\delta\in[0,1]$ is a uniform distributed variable for $J=n_1,\,n_2$. The stationary state is given by \cite{ABR11}
	
	$$(n_1^*,n_2^*)=\Big(\frac{a_2c_1-a_1c_2}{b_2c_1-b_1c_2},\frac{a_2b_1-a_1b_2}{b_1c_2-b_2c_1}\Big),$$
	where $a_1=0.61,\; a_2=0.52,\; b_1=0.4575,\; b_2=0.31,\; c_1=9.5,\; c_2=8.2.$
	
	In Figure \ref{fig3}, we observe islands of high concentration of preys are formed. This reflects the phase separation triggered by preys avoiding predator.
	
	\begin{figure}[pos=!ht]
		\centering
		{\includegraphics[height=2in ,width=6in]{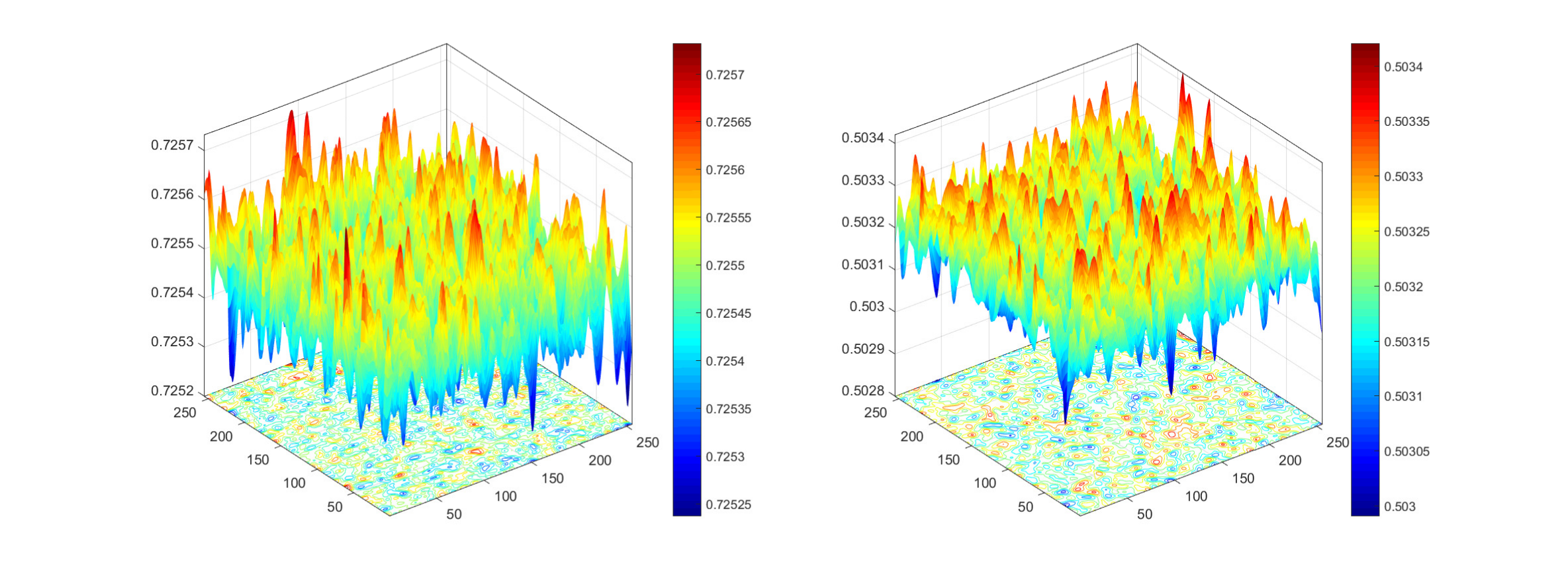}}
		{\includegraphics[height=2in ,width=6in]{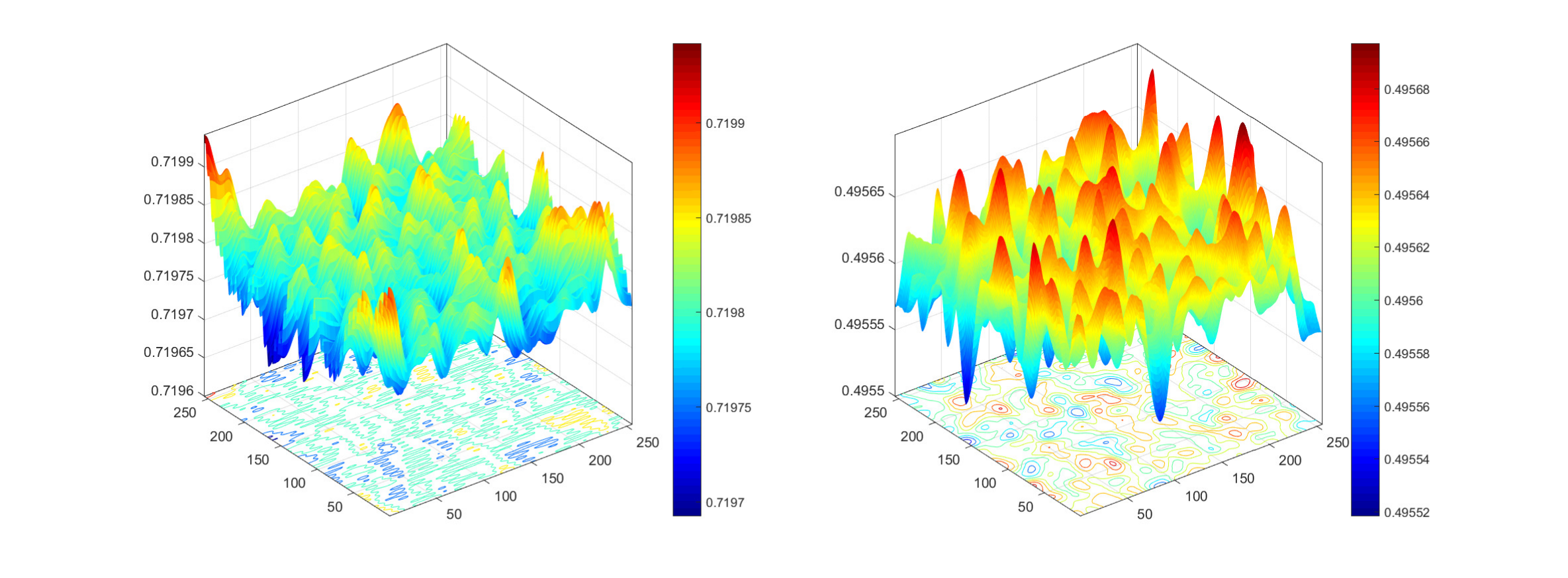}}
		{\includegraphics[height=2in ,width=6in]{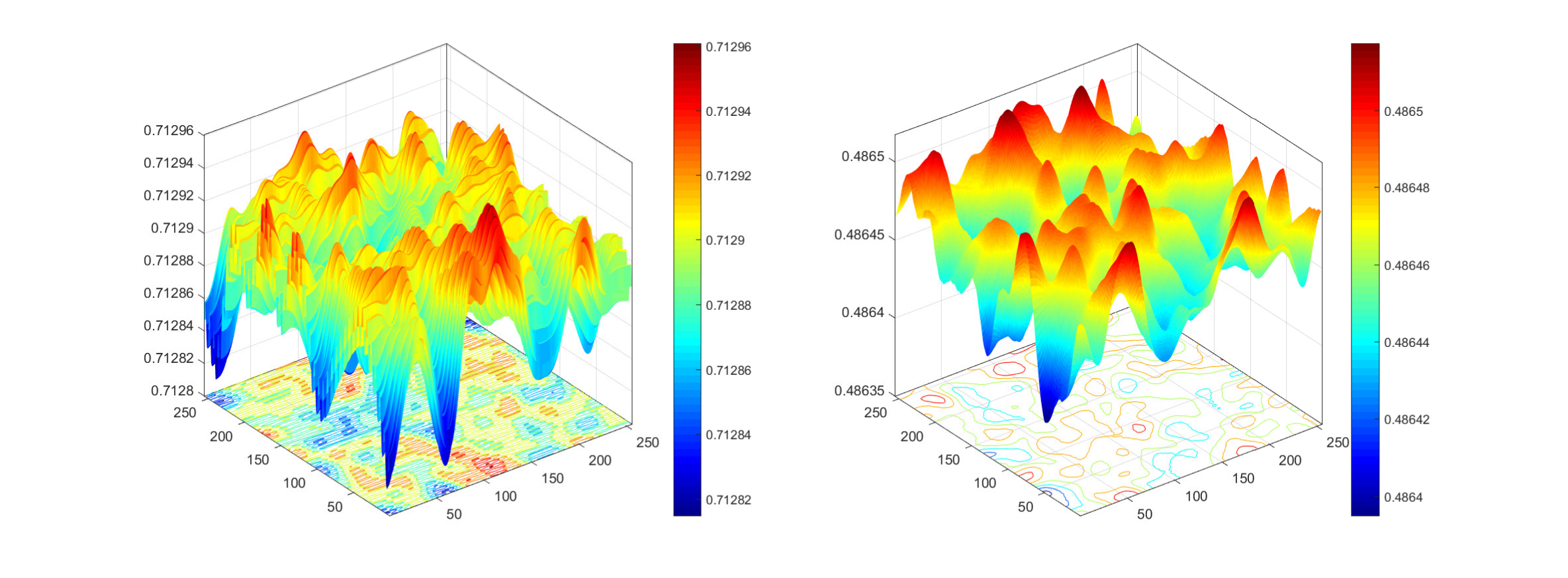}}
		{\includegraphics[height=2in ,width=6in]{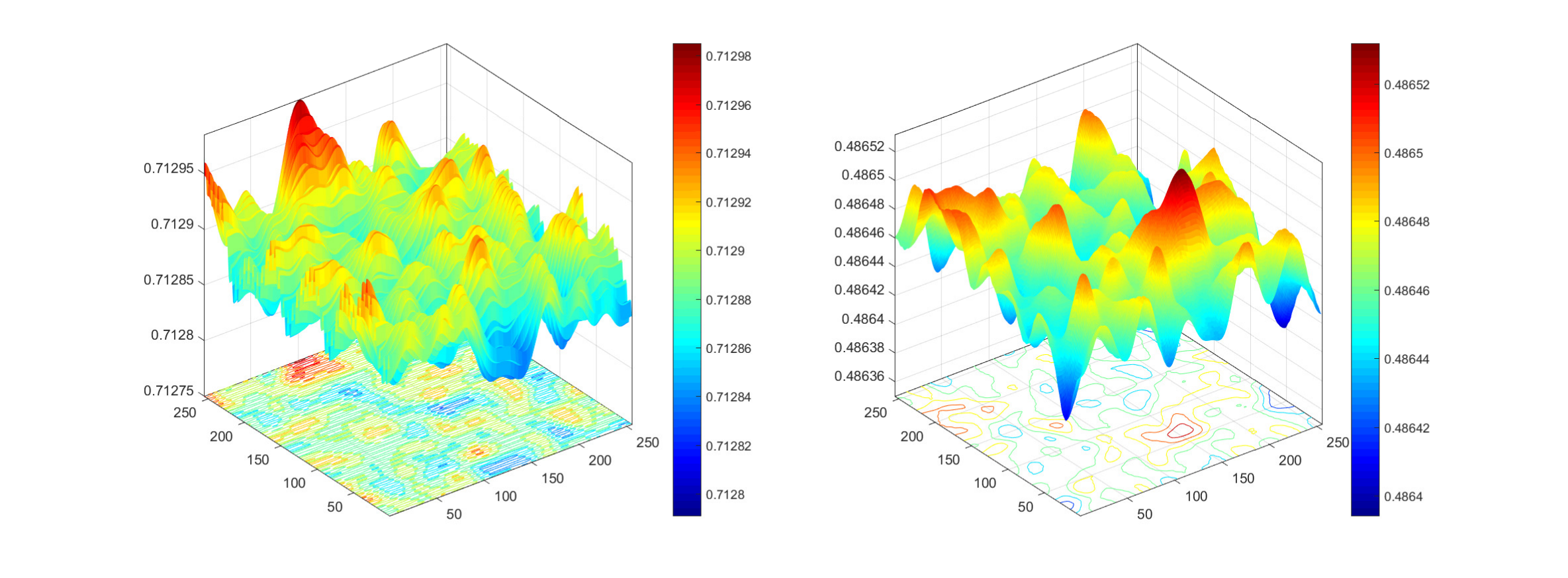}}
		\caption{Example 3: interaction of the predator and prey at different times $t = 0.001,\, 0.001,\, 0.005,\;0.01$}
		\label{fig3}
	\end{figure}

	\subsection{Cross-diffusion with chemicals in the presence of fluid}
	In this subsection, we demonstrate the external action effect on the dynamic the fluid medium, consequently on the evolution of the prey and predator densities. The spatial domain $\Omega$ corresponds to a rectangle $(0, 10) \times (0, 4)$ and contains two obstacles; see Figure \ref{domain}. We consider system \eqref{CrossDiff} with the following initial and boundary conditions:
	
	\begin{equation}\label{NRD-NS}
		\left\{
		\begin{array}{ll}
			\displaystyle U(0,\bx)=U_0,\quad n_1(0,\bx)=n_{1,0},\quad  n_2(0,\bx)=n_{2,0},&\hbox{in}\; \Omega,\\
			\displaystyle \frac{\partial n_1}{\partial \eta}=\frac{\partial n_2}{\partial \eta}=0,&\hbox{on}\; \Gamma_1\cup\Gamma_2\cup\Gamma_3\cup\Gamma_4,\\
			\displaystyle  n_1= n_2= 0,\quad U(x,y)=(u=0,v=0)^T,&\hbox{on}\; \Gamma_5\cup\Gamma_6,
			\\
			\displaystyle   U(x,y)=\Big(\frac{\partial u}{\partial \eta}=0\,,\,0\Big)^T,&\,\hbox{on}\;\Gamma_1\cup\Gamma_3,
			\\
			\displaystyle   U(x,y)=\Big(0,\,\frac{\partial v}{\partial \eta}=0\Big)^T,&\hbox{on}\; \Gamma_2,\\
			\displaystyle   U(x,y)=(2\,{y}\;(1-{y}),\,0)^T,&\hbox{on}\; \Gamma_4.
		\end{array}
		\right.\end{equation}

	Here, all computations have been implemented using the software package FreeFem++ \cite{HF13}. The code uses a finite element method based on the weak formulation of cross-diffusion with chemicals system \eqref{CrossDiff} in an iterative manner as follows:
	\begin{itemize}
		\item Solve Navier-Stokes equations and the incompressibility condition $\eqref{CrossDiff}_5$  with the Characteristic
		Galerkin method. We mention that we have used a classical Taylor-Hood element technique, i.e. the fluid velocity $U$ is approximated by $P2$ finite elements and the pressure $p$ is approximated by $P1$ finite elements.
		\item Approximate the densities $n_1$ and $n_2$ by $P2$ finite elements and solve firstly Eq. $\eqref{CrossDiff}_1$, then
		Eq. $\eqref{CrossDiff}_2$ and finally Eqs. $\eqref{CrossDiff}_{3,4}$. We mention that we have used UMFPACK package and $\theta$-scheme with $\theta=0.49$
	\end{itemize}
	\begin{figure}[pos=!ht]
		\centering
		{\includegraphics[height=2.5in ,width=6in]{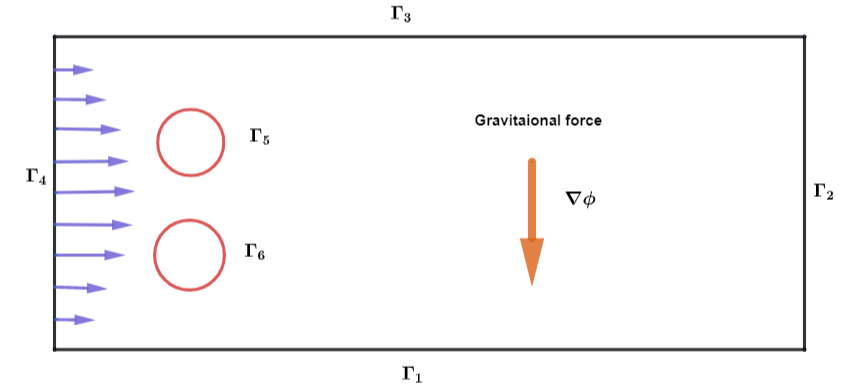}}
		\caption{The spatial domain with boundary conditions}
		\label{domain}
	\end{figure}
	
	We recall that $\nabla \phi= V_s(\rho_s - \rho_f ) g\vec{z}$ , where $V_s$ and $\rho_s$ are, respectively, the volume and the density of species, $\rho_f$ is the fluid density, and $g$ is the gravitational force. The vector $-\nabla \phi$is the resultant of gravitational forces
	$(\vec{P} = -\rho_s V_s g\vec{z} )$ and the Archimedes thrust $(\vec{Fa} = \rho_fV_s g\vec{z} ).$ In our tests, the populations are denser than the fluid and therefore a gravitational flow is created in the direction of the vector $-\overrightarrow{z}$. We consider two cases:
	
	In the first case, we illustrate the behavior of cross-diffusion--fluid with chemicals system \eqref{CrossDiff} in the absence of gravitational force; that is, $ \nabla \phi=(0, 0)$.

	In Figure \ref{Fig3}, we display the numerical simulations of the densities $n_1$ and $n_2$ of the two interacting populations and the dynamics of the fluid flow presented by the fluid velocity $U$ and the pressure $p$. Initially, we observe the cross-diffusion effect; that is to say the predator directs its movement towards the region occupied by the prey, while the prey moves toward the area where the predator is not located. Next, we notice that the prey and the predator are transported in the direction of the fluid. Moreover, we observe that the fluid flow is not influenced by the presence of the populations in the medium; however, it is affected by the presence of the obstacle in the domain.
	
	In the second case, we assume the presence of gravitational force; that is $\nabla \phi=(0,-1)$. Thus, we obtain the strong coupling
	system \eqref{CrossDiff}.
	In Figure \ref{Fig4}, we provide the numerical simulations of the two densities and the dynamics of the fluid flow  presented by the fluid velocity $U$ and the pressure $p$.
	Clearly, we observe that the densities and the fluid are influenced by the presence of gravitational force. In addition, we observe also the effect of the presence of the two obstacles.
	
	\begin{figure}[pos=!ht]
		\centering
		\subfigure{\includegraphics[height=1.2in ,width=1.5in]{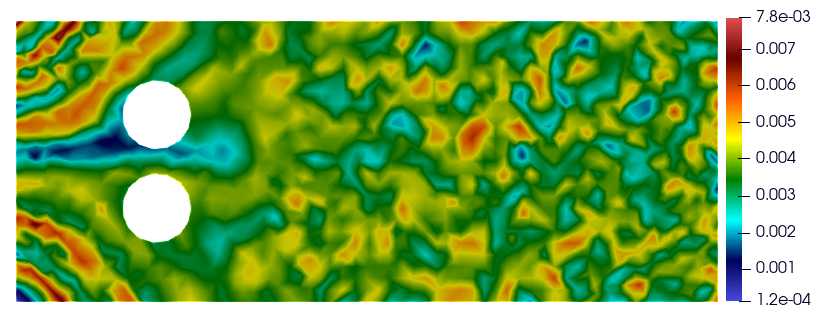}}
		~\subfigure{\includegraphics[height=1.2in ,width=1.5in]{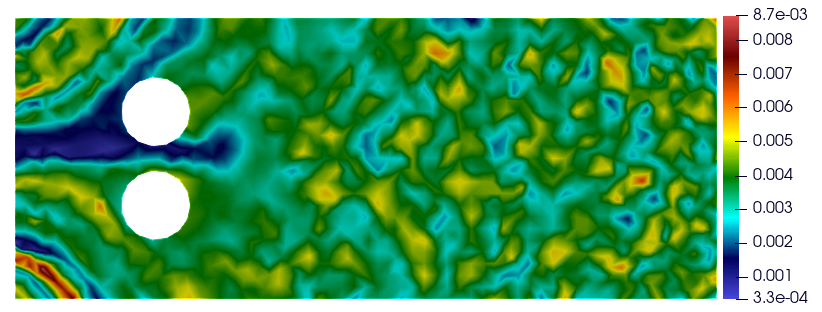}}
		~\subfigure{\includegraphics[height=1.2in ,width=1.5in]{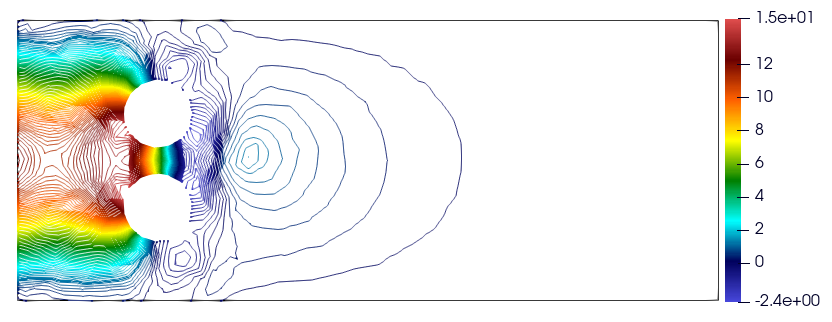}}
		~\subfigure{\includegraphics[height=1.2in ,width=1.5in]{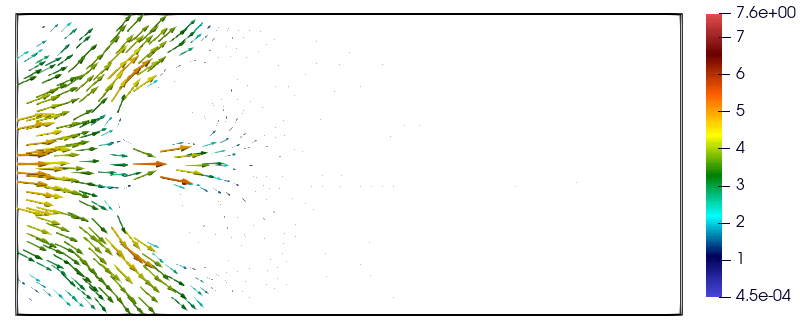}}
		
		\subfigure{\includegraphics[height=1.2in ,width=1.5in]{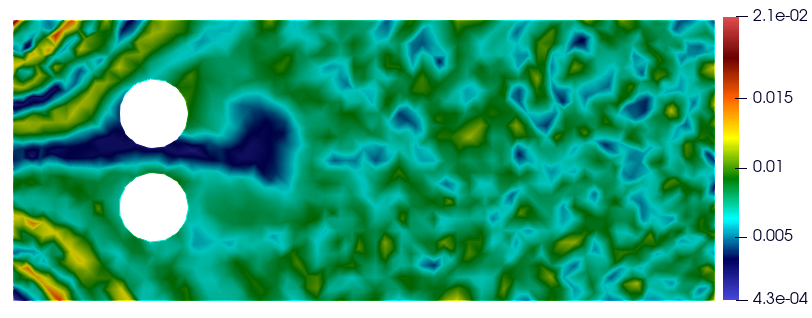}}
		~\subfigure{\includegraphics[height=1.2in ,width=1.5in]{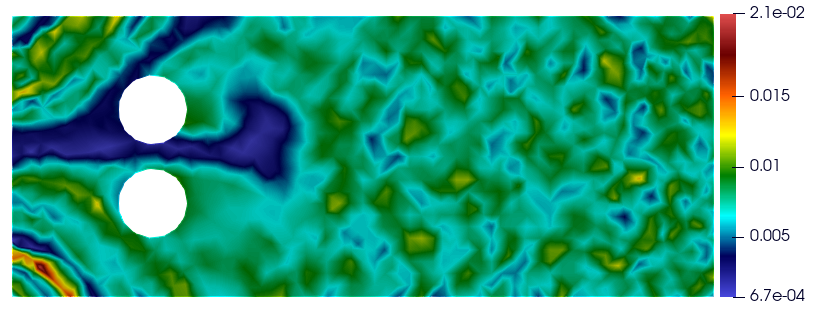}}
		~\subfigure{\includegraphics[height=1.2in ,width=1.5in]{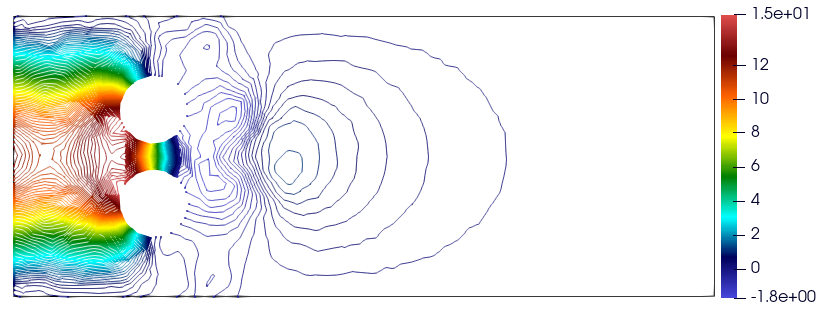}}
		~\subfigure{\includegraphics[height=1.2in ,width=1.5in]{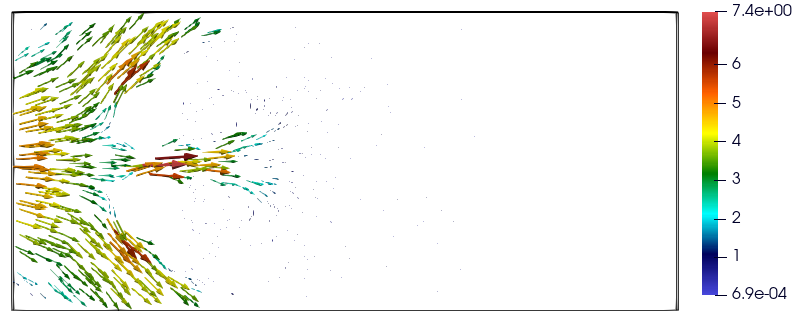}}
		
		\subfigure{\includegraphics[height=1.2in ,width=1.5in]{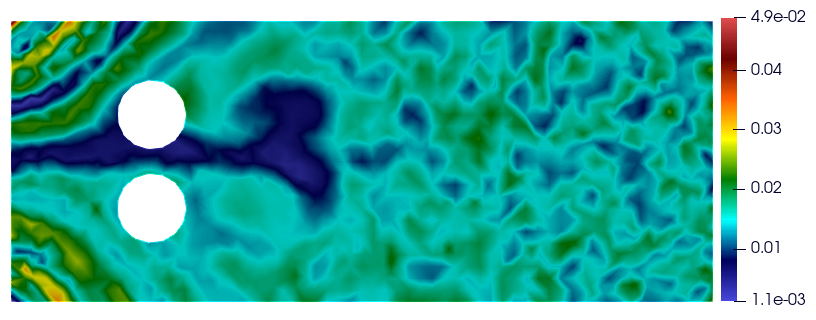}}
		~\subfigure{\includegraphics[height=1.2in ,width=1.5in]{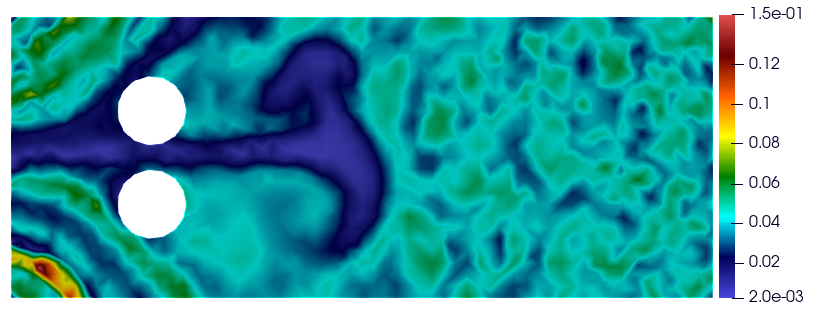}}
		~\subfigure{\includegraphics[height=1.2in ,width=1.5in]{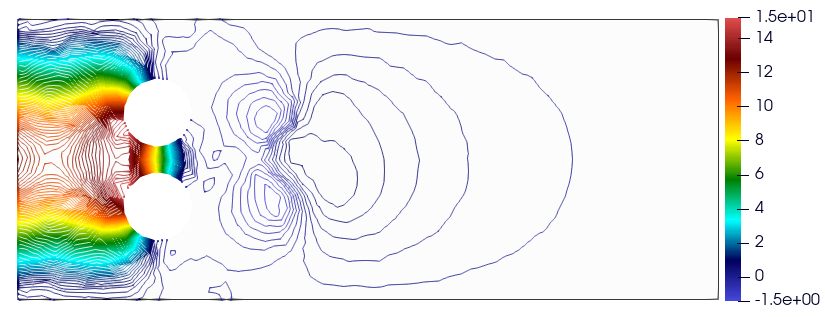}}
		~\subfigure{\includegraphics[height=1.2in ,width=1.5in]{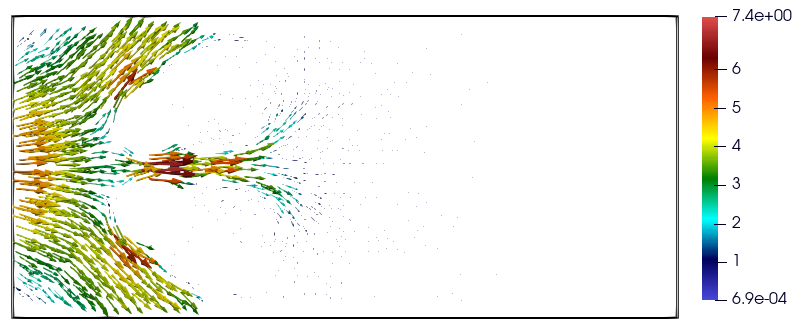}}
		
		\subfigure{\includegraphics[height=1.2in ,width=1.5in]{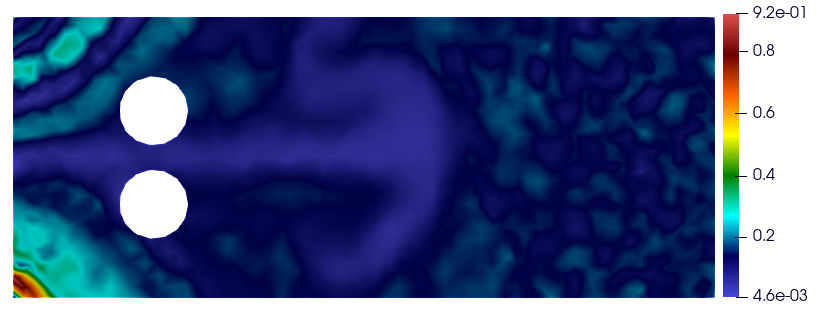}}
		~\subfigure{\includegraphics[height=1.2in ,width=1.5in]{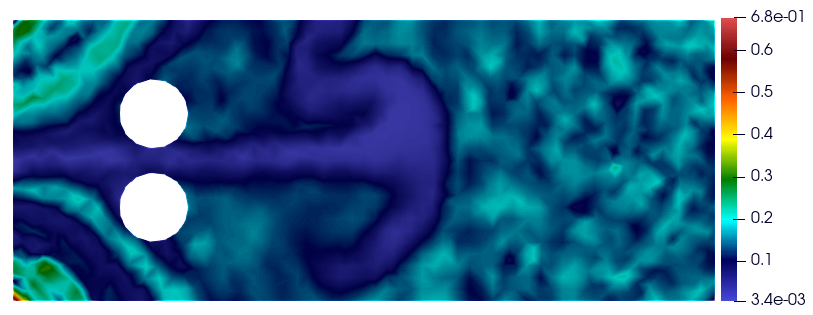}}
		~\subfigure{\includegraphics[height=1.2in ,width=1.5in]{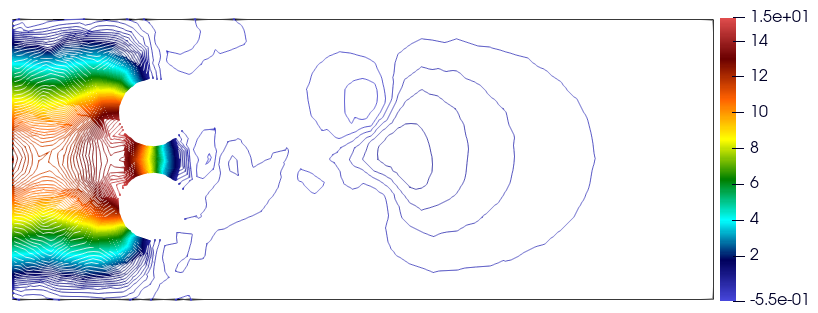}}
		~\subfigure{\includegraphics[height=1.2in ,width=1.5in]{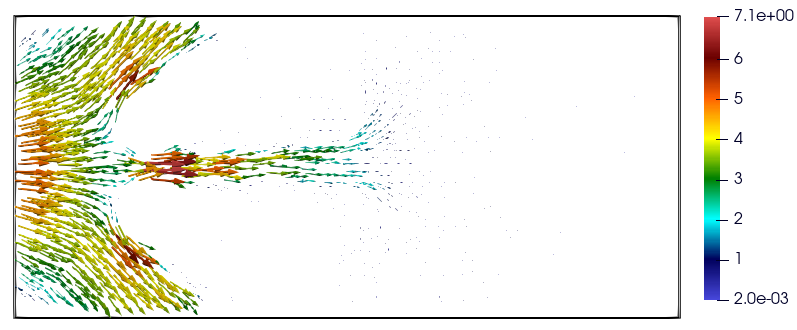}}
		\caption{Test 1: Evolution of the two interacting populations and snapshot of the fluid velocity and the pressure in the case $\nabla \phi=(0,0)$ at different times $t = 5, 7,\, 10,\, 15$.}
		\label{Fig3}
	\end{figure}
	
	\begin{figure}[pos=!ht]
		\subfigure{\includegraphics[height=1.2in ,width=1.5in]{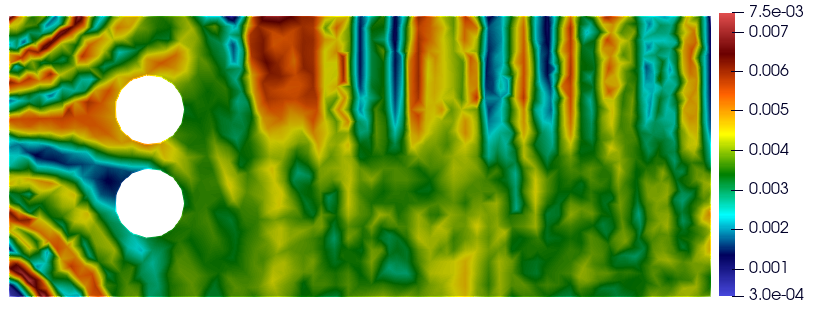}}
		~\subfigure{\includegraphics[height=1.2in ,width=1.5in]{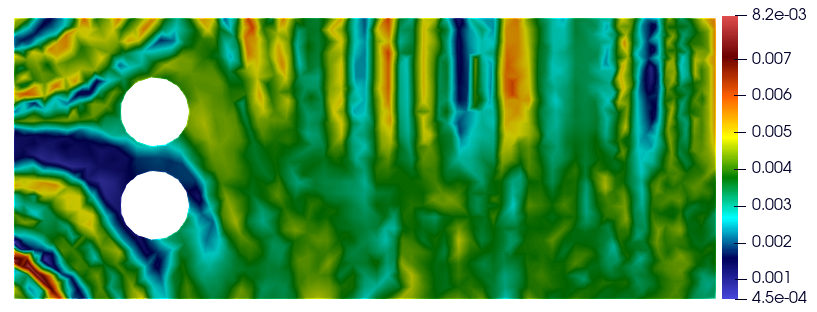}}
		~\subfigure{\includegraphics[height=1.2in ,width=1.5in]{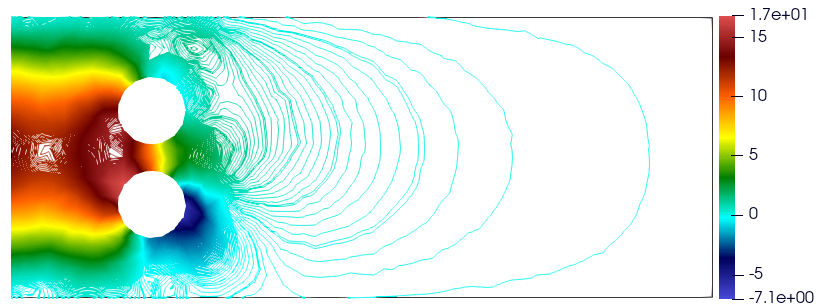}}
		~\subfigure{\includegraphics[height=1.2in ,width=1.5in]{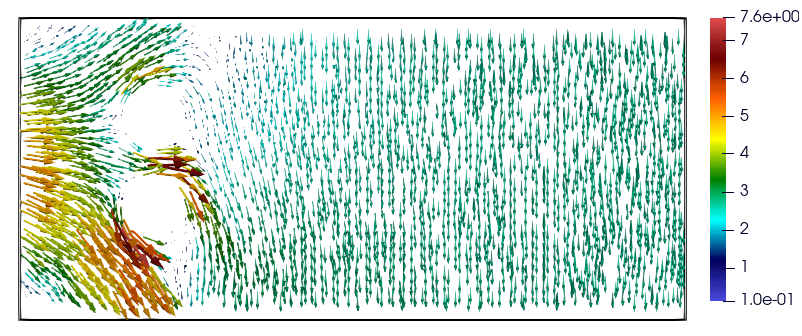}}
		
		\subfigure{\includegraphics[height=1.2in ,width=1.5in]{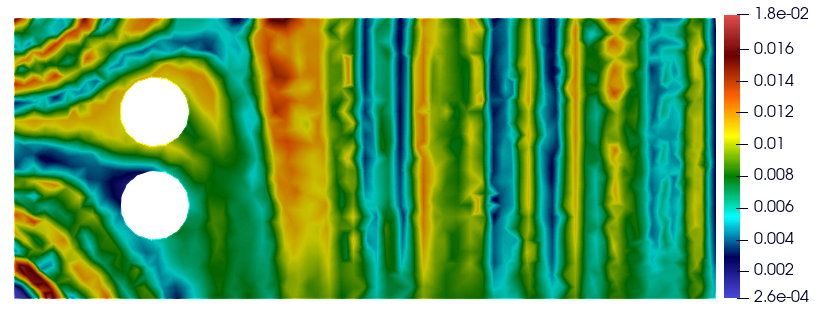}}
		~\subfigure{\includegraphics[height=1.2in ,width=1.5in]{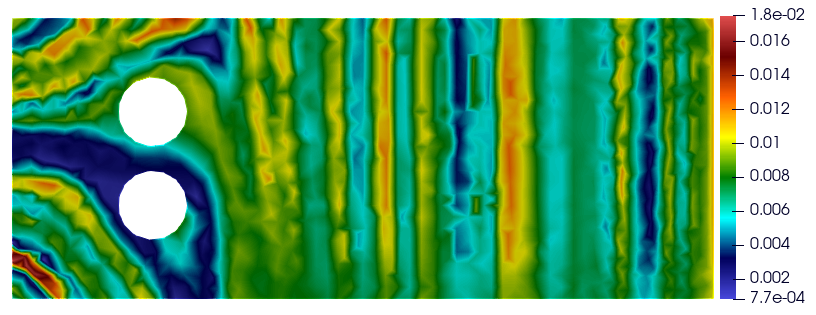}}
		~\subfigure{\includegraphics[height=1.2in ,width=1.5in]{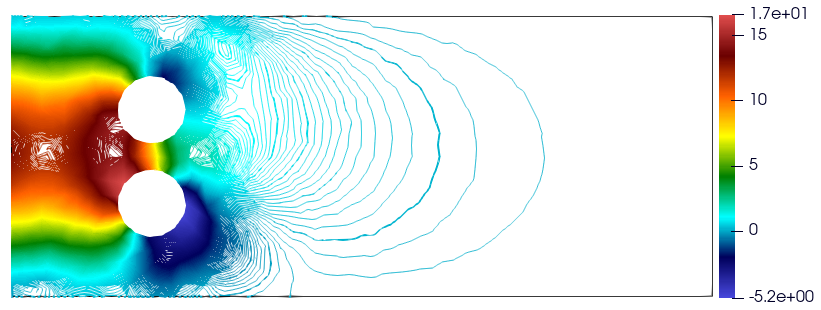}}
		~\subfigure{\includegraphics[height=1.2in ,width=1.5in]{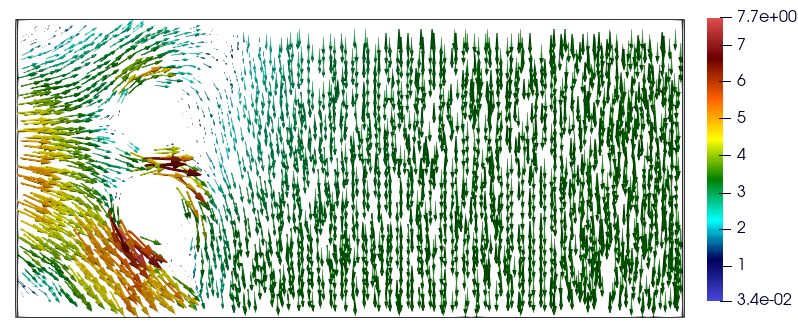}}
		
		\subfigure{\includegraphics[height=1.2in ,width=1.5in]{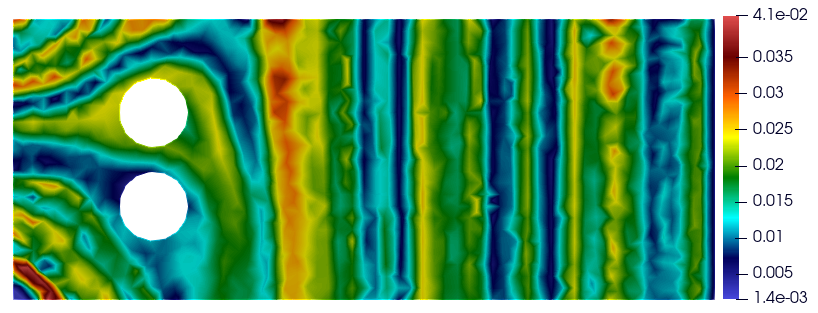}}
		~\subfigure{\includegraphics[height=1.2in ,width=1.5in]{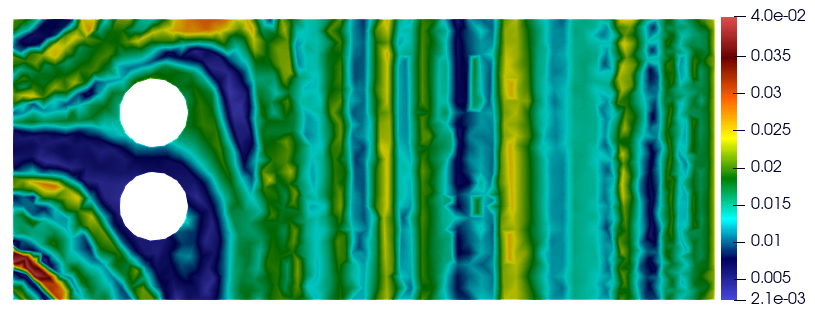}}
		~\subfigure{\includegraphics[height=1.2in ,width=1.5in]{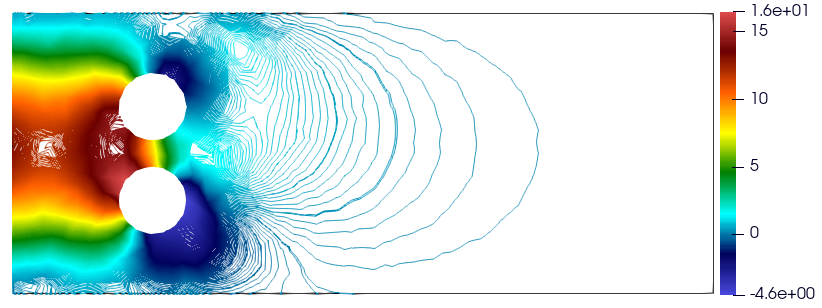}}
		~\subfigure{\includegraphics[height=1.2in ,width=1.5in]{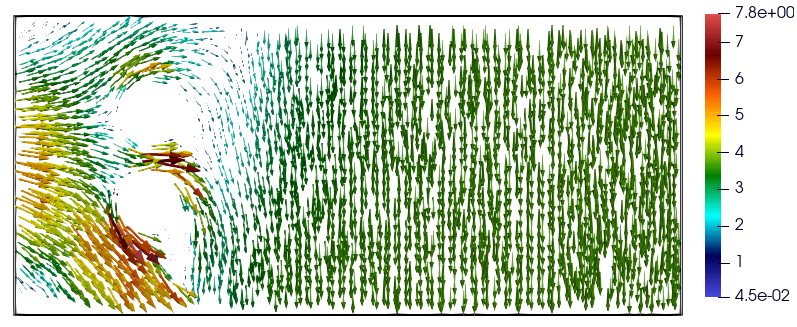}}
		
		\subfigure{\includegraphics[height=1.2in ,width=1.5in]{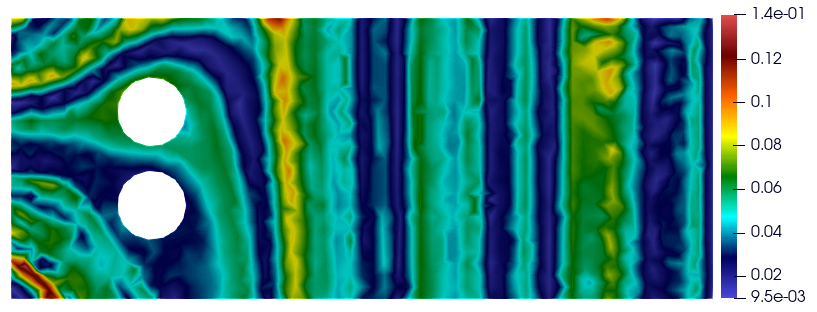}}
		~\subfigure{\includegraphics[height=1.2in ,width=1.5in]{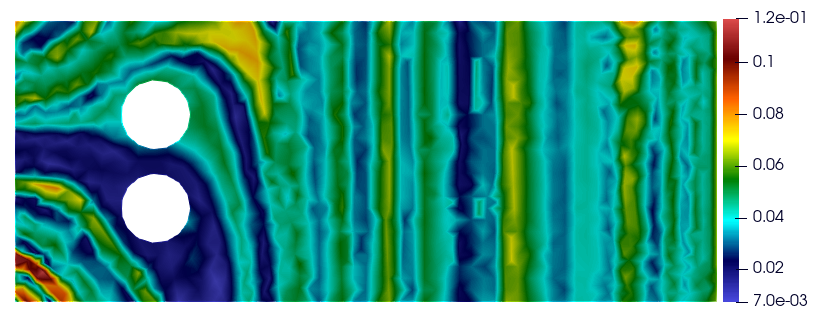}}
		~\subfigure{\includegraphics[height=1.2in ,width=1.5in]{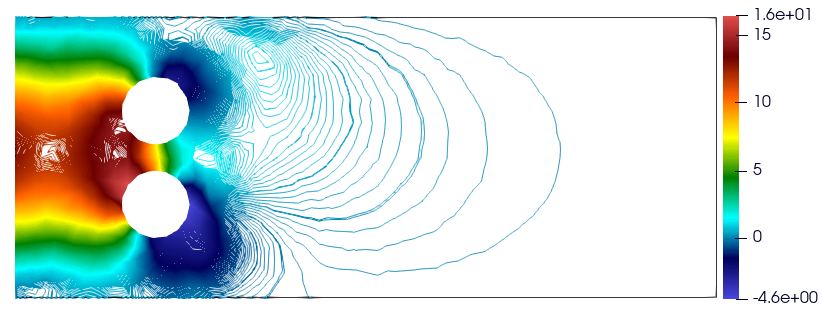}}
		~\subfigure{\includegraphics[height=1.2in ,width=1.5in]{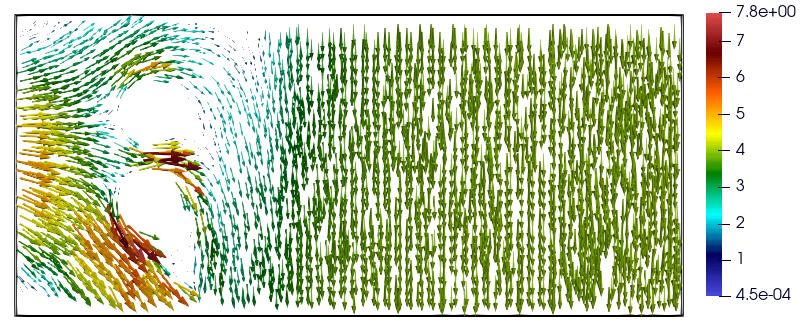}}
		\caption{Test 2: Evolution of the two interacting populations and snapshot of the fluid velocity and the pressure in the case $\nabla \phi=(0,-1)$ at different times $t = 3, 5,\, 7,\, 15$.}
		\label{Fig4}
	\end{figure}

	\section{Conclusion and perspectives}
	
	In this paper, a nonlinear chemotaxis--fluid system with chemical terms describing two interacting species living in a Newtonian fluid governed by the incompressible Navier-Stokes equations has been proposed. The existence of weak solutions of the proposed macro-scale system has been proved. The proof is based on Schauder fixed-point theory, a priori estimates, and compactness arguments. This system was derived from a new nonlinear kinetic--fluid model according to multiscale approach based on the micro-macro decomposition method.
	Several numerical simulations in two dimensional space were provided. Specifically, we showed that prey has a tendency to keep away from predator and at the same time predator has a tendency to get closer to prey. In addition, the phenomenon of pattern formation and the effects of external forces (gravity and spatial domains with two obstacles) on the dynamics of fluid flow and on the behavior of the predator-prey were demonstrated.
	
	Locking ahead, a possible perspective consists in extending the proposed macro-scale system to multiple species (e.g. three species as in \cite{BRS20}), improving our deterministic system to a stochastic system to take into account the environmental noise. Another interesting development would be the numerical analysis of the multiscale micro-macro decomposition method in two-dimensional space, for instance see \cite{BTZ22}. 

\section*{Acknowledgment} This work was done while MB visited ESTE of Essaouira at the University of Cadi Ayyad, Morocco, and he is grateful for the hospitality.

	\bibliographystyle{cas-model2-names}
	\bibliography{BIB}

\end{document}